\documentclass[11pt]{amsart}
\usepackage{tgpagella}
\usepackage{euler}
\usepackage[T1]{fontenc}
\usepackage{amssymb}
\usepackage{amsmath,amsthm}
\usepackage[hidelinks]{hyperref}
\usepackage{amsthm}
\usepackage[english]{babel}
\usepackage{mathrsfs}
\usepackage{eucal}
\usepackage{calligra}

\newtheorem{main}{Theorem}

\newtheorem{mcor}[main]{Corollary}

\newtheorem{theorem}{Theorem}[section]
\newtheorem{lem}[theorem]{Lemma}
\newtheorem{prop}[theorem]{Proposition}

\newtheorem{cor}[theorem]{Corollary}

\theoremstyle{definition}
\newtheorem{definition}[theorem]{Definition}
\newtheorem{notation}[theorem]{Notation}
\newtheorem{Remarks}[theorem]{Remarks}

\newtheorem{claim}[theorem]{Claim}
\def\ca{\curvearrowright}
\def\ra{\rightarrow}

\def\la{\lambda}
\def\La{\Lambda}

\def\bb{\mathbb}

\def\om{\omega}
\def\ve{\varepsilon}

\def\Sg{\Sigma}

\def\bb{\mathbb}
\def\g{\gamma}
\def\si{\sigma}
\def\G{\Gamma}
\def\Cal{\mathcal}

\numberwithin{equation}{section}

\DeclareMathOperator*{\PSL}{PSL}

\theoremstyle{remark}

\newcommand{\set}[1]{ \{ #1 \} }
\newcommand{\norm}[1]{\| #1 \|}

\newcommand{\generator}[1]{\langle #1 \rangle}

\begin{document}

\title[W$^*$-Rigidity for Products of Hyperbolic Groups]{W$^*$-Rigidity for the von Neumann Algebras of Products of Hyperbolic Groups}
\author[I. Chifan]{Ionut Chifan}
\address{Department of Mathematics, The University of Iowa, 14 MacLean Hall, IA  
52242, USA}
\email{ionut-chifan@uiowa.edu}
\thanks{I.C.\ was supported by NSF Grant DMS \#1301370.}

\author[R. de Santiago]{Rolando de Santiago}
\address{Department of Mathematics,The University of Iowa, 14 MacLean Hall, IA  
52242, USA}
\email{rolando-desantiago@uiowa.edu}
\thanks{R.dS.\ was supported in part by GAANN fellowship grants \#P200A100028 and \#P200A120058 and by a Sloan Center minigrant.}

\author[T. Sinclair]{Thomas Sinclair}
\address{Department of Mathematics, Purdue University, 150 N University St, West Lafayette, IN 47907-2067, USA}
\email{tsincla@purdue.edu}


\date{\today }
\dedicatory{}

\begin{abstract} We show that if $\G = \G_1\times\dotsb\times \G_n$ is a product of $n\geq 2$ non-elementary ICC hyperbolic groups then any discrete group $\La$ which is $W^*$-equivalent to $\G$ decomposes as a $k$-fold direct sum exactly when $k=n$. This gives a group-level strengthening of Ozawa and Popa's unique prime decomposition theorem by removing all assumptions on the group $\La$. This result in combination with Margulis' normal subgroup theorem allows us to give examples of lattices in the same Lie group which do not generate stably equivalent II$_1$ factors.

\end{abstract}

\maketitle


\section{Introduction}

One of the most basic and intractable questions in the study of von Neumann algebras asks how much algebraic information about a (discrete) group $\G$ can detected at the level of the associated group von Neumann algebra $L(\G)$.  For instance, whether the number of generators is an isomorphism invariant among the class of free group factors remains one of the foremost open problems in the field, dating from the seminal work of Murray and von Neumann. Connes' celebrated classification of injective factors \cite{Co76} implies that almost no algebraic information can be recovered from $L(\G)$ when $\G$ is amenable. Even in the non-amenable case work of Dykema \cite{Dy93} shows that much information is lost in general for free products of amenable groups. However, a bold conjecture of Connes from the 1980s asserts that in the case that $\G$ is extremely non-amenable (an ICC property (T) group) that the algebraic structure of $\G$ can essentially be completely recovered from $L(\G)$, i.e., if $L(\G)$ is isomorphic to $L(\La)$ then $\La$ must be isomorphic to $\G$. This phenomena has been termed ``W$^*$-superrigidity'' by Sorin Popa. The first examples of W$^*$-superrigid groups were discovered by Ioana, Popa, and Vaes \cite{IPV10} using Popa's deformation/rigidity theory \cite{Po06}. 

Groups with property (T) belong to a much broader class of ``rigid'' groups which also includes direct products of non-amenable groups, irreducible lattices in higher-rank Lie groups, as well as mapping class groups of surfaces of suitable high genus and braid groups. Such groups have been the basis of many remarkable classification results in the last decade in the theory of von Neumann algebras as well as the theory of measured equivalence relations. In this paper, we will focus on class of groups which arise as finite direct products of ICC hyperbolic groups, where essentially the presence of nontrivial commutation exactly between the factor groups provides a very strong rigidity property to exploit. This particular case of ``product rigidity'' has already been used to great effect in the fundamental work of Ozawa and Popa \cite{OP03} where such group factors are shown to admit a unique tensor product decomposition into prime factors. In particular this shows that if $\G = \G_1\times\dotsb\times\G_m$ and $\La_1\times\dotsb\times\La_n$ are products of ICC hyperbolic groups $\G_1,\dotsc, \G_m, \La_1, \dotsc, \La_n$ then if $L(\G)\cong L(\La)$ it must be the case that $m=n$.

Building on the techniques from \cite{OP03} and \cite{Io11} we show in this paper that the group von Neumann algebra $L(\G)$ of such a group completely remembers the product structure of the generating group $\G$; that is, any isomorphic group von Neumann algebra $L(\La)$ admits its unique prime decomposition at the level of the generating group $\La$ with no \emph{a priori} assumption of a product structure on that group. This appears to be the strongest type of rigidity possible for the class of nontrivial products of ICC hyperbolic groups. The key insight into upgrading Ozawa and Popa's result comes from an ultraproduct-based ``discretization'' technique for group von Neumann algebras which was introduced by Adrian Ioana \cite[Theorem 3.1]{Io11}. This allows one to transfer the existence of ``large'' commuting subgroups in $\G$ through (stable) W$^*$-equivalence to the target group $\La$. The argument then proceeds through a series intertwining lemmas using (strong) solidity results from \cite{CSU11, PV12} to show this implies the existence of commuting, generating subgroups, after which Ozawa and Popa's unique prime decomposition theorem obtains.

Before stating the main result, we describe a more general class of groups for which product rigidity holds. Recall that a countable discrete group $\G$ is said to belong to the class $\mathcal S$ of Ozawa \cite{Oz05} if there is a sequence of maps $\mu_n: \G\to \ell^1(\G)$, where $\|\mu_n(\g)\| = 1$ and $\lim_n\|\mu_n(s\g t) - \la_s\mu_n(\g)\|=0$ for all $\g,s,t\in \G$. By \cite{PV12}, $\G$ belongs to the class $\mathcal S$ if and only if $\G$ belongs to the class $\mathcal{QH}_{reg}$ of \cite{CS11} and is exact. 

We denote by $\mathcal S_{nf}$ the class of all non-amenable, ICC groups in $\mathcal S$.

\begin{main}\label{main3} Let $\G_1,\G_2,\ldots ,\G_n\in {\mathcal S}_{nf}$ be weakly amenable groups with $\G=\G_1\times \G_2\times\cdots\times\G_n$ and denote by $M=L(\G)$. Let $t>0$ be a scalar and let $\La$ be an arbitrary group such that $M^t =L(\La)$.  Then one can find subgroups $\La_1,\La_2,\ldots,\La_n<\La$ with $\La_1\times \La_2\times \cdots \times \La_n=\La$, scalars $t_i>0$ with $t_1t_2\cdots t_n=t$, and a unitary $u\in M$ such that $uL(\La_i)u^*= {L(\G_i)}^{t_i}$ for all $1\leq i\leq n$.
\end{main}

It seems very likely that the assumption of weak amenability is unnecessary: we show that when $\G$ is a product of two groups in the class $\mathcal S_{nf}$ that this is indeed the case.

\begin{mcor}\label{main1}  Let $\G_1,\G_2\in \mathcal S_{nf}$, and denote $\G=\G_1\times \G_2$. Let  $\La$ be an arbitrary group such that $M=L(\G)=L(\La)$. Then one can find subgroups $\La_1,\La_2<\La$ with $\La_1\times \La_2=\La$, a scalar $s>0$, and a unitary $v\in M$ such that $vL(\La_1)v^*= L(\G_1)^s$ and $vL(\La_2) v^*= L(\G_2)^{1/s}$. 
\end{mcor}

These results maybe seen as the group von Neumann algebraic analog of Monod and Shalom's \cite[Theorem 1.10]{MS06} famous orbit equivalence rigidity theorem for products of groups in the class $\mathcal C_{reg}$, though the relation between the two results is imperfect and the proofs do not seem in any precise way to depend on a common framework. One salient point of contrast is the need in \cite{MS06} for a mild mixingness assumption on the target action, while in the case of the above theorem the ICC condition on the target group (corresponding to ``plain ergodicity'') suffices; however, this is likely accounted for by the fact that in orbit equivalence one is working over a parameter space on the group algebra. Secondly, Monod and Shalom are able to deduce honest isomorphism of the groups while in the above theorem the identification of the product factors up to stable isomorphism is sharp. This is because for any pair of II$_1$ factors $M$ and $N$ and any $s>0$ we have canonically that $M \bar\otimes N\cong M^{1/s}\bar\otimes N^s$; thus, using Voiculescu's scaling formula for free group factors \cite{Vo90} we have that $L(\mathbb F_2)\bar\otimes L(\mathbb F_9) \cong L(\mathbb F_2)^{1/2}\bar\otimes L(\mathbb F_9)^2\cong L(\mathbb F_5)\bar\otimes L(\mathbb F_3)$. Hence, for the class $\mathcal S_{nf}$ in general one cannot hope to further deduce (virtual) isomorphism of the factor groups. 

As a consequence of these results we may apply Margulis' Normal Subgroup Theorem \cite{Ma79, Z84} to deduce indecomposability of group factors of higher-rank irreducible lattices over a product of group factors of groups in the class $\mathcal S_{nf}$.

\begin{mcor}\label{main2} If $\La$ is an irreducible lattice in a higher rank semisimple Lie group, then $L(\La)$ is neither isomorphic to a factor $L(\G_1\times \G_2)$ where $\G_1,\G_2$ are groups in the class $\mathcal S_{nf}$, nor is it isomorphic to a factor of the form $L(\G_1\times\dotsb\times \G_n)$ where each $\G_i\in \mathcal S_{nf}$ and is weakly amenable. 

\end{mcor}
 In particular if $\La = \PSL_2(\mathbb Z[\sqrt{2}])$, then $L(\La)$ is not isomorphic to $L(\mathbb F_2\times \mathbb F_2)$, even though these groups are measure equivalent in the sense of \cite{Fu99a}. These add new natural examples to the ones found earlier \cite{CI11}.
 
As a last remark we believe that the arguments below actually give a slightly stronger version of Theorem \ref{main3}, though we will not pursue this here to avoid additional complexities in the presentation. To introduce a bit of notation, we say that two finite von Neumann algebras $M$ and $N$ are \emph{virtually stably isomorphic} if there exist $s,t>0$ and finite index subalgebras $\mathcal M\subset M^s$ and $\mathcal N\subset N^t$ so that $\mathcal M$ and $\mathcal N$ are $\ast$-isomorphic. With this notation in hand, it should be the case that if $\G = \G_1\times\dotsb\times\G_n$ with all $\G_i$ weakly amenable and belonging to the class $\mathcal S_{nf}$, then for any discrete group $\La$, $L(\La)$ is virtually stably isomorphic to $L(\G)$ if and only if there is a finite index subgroup $\La_0< \La$ and a finite normal subgroup $N\lhd \La_0$ so that $\La_0/N\cong \La_1\times\dotsb\times\La_n$ with $L(\La_i)$ stably isomorphic to $L(\G_i)$, for all $1\leq i\leq n$.
\subsection{Notations} Given a von Neumann algebra 
$M$ we will denote by $\mathscr U(M)$ its unitary group, by $\mathscr P(M)$ the set of all its nonzero projections and by $\mathscr I (M)$ the set of all its nonzero partial isometries. Also we denote by $M_+$ the set of all positive elements and $M^h$ is the set of all selfadjoint elements. All von Neumann algebras inclusions $N\subseteq M$ are assumed unital unless otherwise specified. For any von Neumann subalgebras $P,Q\subseteq M$  we denote by $P\vee Q$ the von Neumann algebra they generate in $M$.

All von Neumann algebras $M$ considered in this article will be tracial, i.e., endowed with a unital, faithful, normal functional $\tau:M\ra \mathbb C$  satisfying $\tau(xy)=\tau(yx)$ for all $x,y\in M$. 

For a countable group $\G$ we denote by $\{ u_\g \,:\, \g\in \G\} \in U(\ell^2(\G))$ its left regular representation given by $u_\g(\delta_\la ) = \delta_{\g\la}$, where $\delta_\la:\G\ra \mathbb C$ is the Dirac mass at $\la$. The weak operatorial closure of the linear span of $\{ u_\g \,:\, \g\in \G\}$ in $B(\ell^2(\G))$ is the so called group von Neumann algebra and will be denoted by $L(\G)$. $L(\G)$ is a II$_1$ factor precisely when $\G$ has infinite non-trivial conjugacy classes (ICC).

Given a group $\G$ and a subset $F\subseteq \G$ we will be denoting by $\langle F\rangle$ the subgroup of $\G$ generated by $F$.

\noindent {\bf Acknowledgements.} We would like to thank R\'emi Boutonnet and Adrian Ioana for useful comments on the manuscript.

\section{Preliminaries}


\subsection{Popa's Intertwining Techniques} More than a decade ago S. Popa introduced  in \cite [Theorem 2.1 and Corollary 2.3]{Po03} a powerful criterion for identifying intertwiners between arbitrary subalgebras of tracial von Neumann algebras, now termed \emph{Popa's intertwining-by-bimodules techniques}.

\begin {theorem}\cite{Po03} \label{corner} Let $(M,\tau)$ be a separable tracial von Neumann algebra and let $P, Q\subseteq M$ be (not necessarily unital) von Neumann subalgebras. 
Then the following are equivalent:
\begin{enumerate}
\item There exist  nonzero projections $ p\in  P, q\in  Q$, a $\ast$-homomorphism $\theta:p Pp\rightarrow qQq$  and a nonzero partial isometry $v\in q Mp$ such that $\theta(x)v=vx$, for all $x\in pPp$.
\item For any group $\mathcal G\subset \mathscr U(P)$ such that $\mathcal G''= P$ there is no sequence $(u_n)_n\subset \mathcal G$ satisfying $\|E_{ Q}(xu_ny)\|_2\rightarrow 0$, for all $x,y\in  M$.
\end{enumerate}
\end{theorem} 

If one of the two equivalent conditions from Theorem \ref{corner} holds then we say that \emph{ a corner of $P$ embeds into $Q$ inside $M$}, and write $P\preceq_{M}Q$. If we moreover have that $Pp'\preceq_{M}Q$, for any nonzero projection  $p'\in P'\cap 1_PM1_P$, then we write $P\preceq_{M}^{s}Q$.

\subsection{Finite Index Inclusions of von Neumann Algebras} If $P\subseteq M$ are II$_1$ factors, then the {\it Jones index} of the inclusion $P\subseteq M$, denoted  $[M:P]$, is  the dimension of $L^2(M)$ as a left  $P$-module. For various basic properties of finite index inclusions of factors we refer the reader to the groundbreaking work of V.F.R. Jones, \cite{Jo81}. Subsequently, there were several generalizations of the finite index notion for an inclusion of arbitrary von Neumann algebras. For instance, M. Pimsner and S. Popa  discovered a ``probabilistic'' notion of index for an inclusion $P\subseteq M$ of arbitrary von Neumann algebras with conditional expectation, which  the case of inclusions of II$_1$ factors coincides with Jones' index, \cite[Theorem 2.2]{PP86}.  

\begin{definition}\cite{PP86} \label{index} For an inclusion $P\subseteq M$ of tracial von Neumann algebras  define  $$\lambda=\inf\;\{\|E_P(x)\|_2^2 \|x\|_2^{-2} \;:\; x\in M_{+}, \; x\not=0\}.$$
The {\it probabilistic index of $P\subseteq M$} is defined as $[M:P]_{PP}=\lambda^{-1}$, with the convention $\frac{1}{0}=\infty$.
\end{definition}

For further use we collect together some basic results from \cite{Jo81,PP86}. 

\begin{theorem}  \cite{Jo81,PP86}\label{basic} Let $P\subseteq M$ be an inclusion of tracial von Neumann algebras. Then the following hold:
\begin{enumerate}
\item \label{same} If $P\subseteq M$ are II$_1$ factors then  $[M:P]_{PP}=[M:P]$;
\item If $[M:P]_{PP}<\infty$ and $p\in P$ is a projection then $[pMp:pPp]_{PP}<\infty$;
\item \label{fdim} If P is II$_1$ factor and $[M:P]_{PP}<\infty$ then $P'\cap M$ is finite dimensional;
\item If $P\subseteq M$ and $Q\subseteq R$ are II$_1$ factors then $[M:P][R:Q]=[M\bar\otimes R: P\bar\otimes Q]$;
\item If $P\subseteq M$ are II$_1$ factors with $[M:P]<\infty$ then there exists a finite Pimsner-Popa basis for the inclusion $P\subseteq M$. 
\item If $P\subseteq M$ are II$_1$ factors then $\dim_{\mathbb C}(P'\cap M)\leq [M:P]$. 
\end{enumerate}
\end{theorem}

In the remaining part of the section we prove several technical results which will be used in the sequel. Some of them are probably well-known, but as we are unable to find the proofs in the literature, we will include them for the reader's convenience.

\begin{lem} \label{decay1'} Let $(M, \tau)$ be a finite von Neumann algebra together with a projection $e\in M$ and a subset $S\subseteq \mathscr U(M)$. Given $\ve>0$, there exists $\eta> 0$ so that if there exists a function $\phi: S\ra \mathbb R_+$ satisfying the following properties:
\begin{enumerate}
 \item \label{0.1.1'}  $\tau(exex^*)\leq \eta +\phi(x)$, for all $x\in S$;
 \item \label{0.1.2'}  for every $\delta>0$ and every finite set $F\subset S$ there exists $u\in S$ such that $\phi (u^*y)\leq \delta$ for all $y\in F$.
 \end{enumerate}
then $\tau(e)\leq \ve$.
\end{lem}

The following concise proof, which also produces a much better effective constant $\eta$ than we had originally obtained, was shown to us by R\'emi Boutonnet.

\begin{proof} Fix $\ve>0$, and let $\eta = \ve^2/2$. Using property (\ref{0.1.2'}) inductively for $k\in \mathbb N$ arbitrary we can find elements $\sigma _1,\sigma_2,\ldots, \sigma_k \subset S$ so that that if $e_i= \sigma_i e {\sigma_i}^*$ then 
\begin{equation}\label{8.5.0}\tau(e_ie_j)\leq 2\eta
\end{equation} for all $i\not= j$.

It thus follows from the Cauchy-Schwartz inequality that 
\begin{equation} (k\tau(e))^2 = \left(\tau(\sum_i e_i)\right)^2\leq \sum_{i.j} \tau(e_ie_j) = k\tau(e) + 2k^2\eta
\end{equation}
Setting $X = k\tau$, one has the quadratic inequality $X^2 \leq X + 2k^2\eta$, solving which leads to 
\begin{equation} \tau(e)\leq \frac{1 + \sqrt{1 + 8k^2\eta}}{4k}.
\end{equation}
As $k$ was arbitrary this shows that $\tau(e)\leq \sqrt{2\eta} = \ve$. 

\end{proof}

\begin{cor}\label{decay1}Let $(M, \tau)$ be a finite von Neumann algebra together with a projection $e\in M$ and a subset $S\subseteq \mathscr U(M)$. Assume for every $\ve>0$ there exists a function $\phi_\ve: S\ra \mathbb R_+$ satisfying the following properties:
\begin{enumerate}
 \item \label{0.1.1}  $\tau(exex^*)\leq \ve +\phi_\ve(x)$, for all $x\in S$;
 \item \label{0.1.2}  for every $\delta>0$ and every finite set $F\subset S$ there exists $u\in S$ such that $\phi_\ve (u^*y)\leq \delta$ for all $y\in F$.
 \end{enumerate}
Then $e=0$.
\end{cor}

\begin{proof}
Applying the previous lemma, for every $\ve>0$ we have $\tau(e)\leq\ve$; thus $e=0$. 
\end{proof}

From part (2) in Theorem \ref{basic} finiteness of the index of an inclusion of algebras is preserved under taking corners. Next we establish the converse for certain inclusions of group von Neumann algebras. 

\begin{prop}\label{fin-index1} Let $\Omega\leqslant\Lambda\leqslant \Theta$ be groups.  If there are  $p\in \mathscr P(L(\Omega)), z\in \mathscr P(L(\La)'\cap L(\Theta))$ so that $pz\neq 0$ and  $pL(\Omega) pz\subseteq pL(\Lambda) pz$ admits a finite Pimsner-Popa basis then $[\La :\Omega]<\infty$.   
\end{prop}

\begin{proof} Let $e\in \mathscr Z(L(\Omega))$ be the support projection of $E_{L(\Omega)}(z)$ and notice $z\leq e$. For $t>0$ denote by $e_t=1_{(t,\infty)}(e)\in \mathscr Z(L(\Omega))$ and notice $e_t$ $SOT$-converges to $e$, as $t\ra 0$. By assumption there exist elements $m'_1,m'_2,\ldots,m'_s\in pL(\La) p z$ such that for every $x\in pL(\La) pz$ we have $x=\sum_i E_{pL(\Omega) pz}(xm_i^{'\ast})m'_i$. If we denote by $m_i=e_tm'_ie_t$, this further implies that for every $x\in pe_tL(\La) pe_tz$  we have  
\begin{equation}\label{10001}
x=\sum_i E_{pe_tL(\Omega) pe_tz}(xm^*_i)m_i.
\end{equation} 
Notice that $E_{L(\Omega) pe_tz}(pe_tz ype_t z)= pE_{L(\Omega)}(yz)pE_{L(\Omega)}(e_tz)^{-1}e_t z$ for all $y\in L(\La)$, where $E_{L(\Omega)}(e_tz)^{-1}$ is the inverse of $E_{L(\Omega)}(e_tz)$ under $e_t$. Therefore $\|E_{L(\Omega) pe_tz}(pe_tz ype_t z)\|_2\leq t^{-1}\|E_{L(\Omega)}( yz)\|_2$, for all $y\in L(\La)$. This together with (\ref{10001}) and basic approximations of $m_i$'s further imply that for every $\ve>0$ one can find a constant $c_\ve>0$ and a finite subset $L_\ve\subset \Lambda$ such that for every $x\in L(\Lambda)$ we have 
 \begin{equation}\label{1.1.1}\begin{split}
 \tau((pe_tz)x(pe_tz)x^*)\leq\ve + c_\ve \sum_{s\in L_\ve} \|E_{L(\Omega)}(x u_s)\|^2_{2}. 
\end{split}
\end{equation}
Setting $S= \Omega$ and $\phi_\ve (x)= c_\ve \sum_{s\in L_\ve} \|E_{L(\Omega)}(x u_s)\|^2_{2}$ we see (\ref{1.1.1}) shows that property (\ref{0.1.1}) in Corollary \ref{decay1} is satisfied. 

To finish, assume by contradiction $[\Lambda:\Omega]=\infty$. Since we have infinitely many representatives of left cosets of $\Omega$ in $\Lambda$ then for every finite subset $F\subset \Omega$  there exists $\la \in \La$ such that $E_{L(\Omega)}(u^*_\la u_{\si})=0$, for all $\si \in F$. This further shows $\phi_\ve$ also satisfies (\ref{0.1.2}) in Corollary \ref{decay1} and hence $pe_tz=0$. Since this holds for every $t>0$ and $e_t$ $SOT$-converges to $e\geq z$ we get $pz=0$, which is a contradiction.  Hence $[\Lambda: \Omega]<\infty$.     \end{proof}

\begin{cor}\label{diffcorn} For every infinite group $\La$ and every $p\in \mathscr P(L(\La))$ the von Neumann algebra $pL(\La) p$ is diffuse.  
\end{cor}
\begin{proof}
Assuming otherwise, there exists a projection $0\neq q\in \mathscr Z(L(\La))$ so that $L(\La) q=\mathbb C q$. By Proposition \ref{fin-index1} this further implies $\La$ is a finite, contradicting the hypothesis.
\end{proof}


\section{The Transfer of the Commutation Relation through $W^*$-Equivalence}

In this section we prove the first crucial step towards Theorem \ref{main1}, that any group in the $W^*$-equivalence class of a product of groups in $\mathcal S_{nf}$ must itself contain two commuting non-amenable subgroups. In order to do so, we make use of a powerful ``discretization'' technique for group von Neumann algebras introduced by A. Ioana in \cite{Io11}. Since we follow an approach very similar to \cite[Theorem 3.1]{Io11}, we include a proof addressing only the novel points. To be able to state the result properly we need to recall some previous results and introduce notation.

\begin{theorem}\cite[Theorem 3.2]{CS11} Assume $\G$ is an exact non-amenable group together with $\mathcal G$ a family of subgroups for which there exits a weakly-$\ell^2$ representation $\pi$ such that $\mathcal {RA}(\G, \mathcal G, \pi)\neq \emptyset$ (equivalently, $\G$ is bi-exact with respect to the family $\mathcal G$ in the sense of \cite[Definition 15.1.2]{BO08}). Let $\G\ca B$ be a trace preserving action on a finite amenable von Neumann algebra and denote by $M=B\rtimes \G$ the corresponding crossed product von Neumann algebra. For every projection $p\in M$ and every weakly compact embedding $Q\subseteq pMp$ we have either $Q\preceq_M B\rtimes \Sigma$, for some $\Sigma\in \mathcal G$ or $\mathscr N_{pMp}(Q)''$ is amenable.   
\end{theorem}

\begin{notation} Given $\G=\G_1\times \G_2\times \cdots\times \G_n$, for every subset $F\subseteq \{1,\ldots,n\}$ we denote by $ \G_F$ the subgroup of in $\G$ consisting of elements whose $i$-th components are trivial, for all $i\in \{1,2,...,n\}\setminus F$. We denote by $\hat \G_{F}=\G_{\{1,2,...,n\}\setminus F}$ and for brevity $\hat\G_j= \G_{\{1,2,...,n\} \setminus \{j\}}$. Notice that $\hat{\hat {\G}}_F=\G_F$. 
\end{notation}

In the spirit of \cite[Theorem 3.1]{Io11} we prove the following theorem 

\begin{theorem}\label{comm1} For $n\geq 2$ let $\G_1, \G_2, \ldots, \G_n \in \mathcal S_{nf}$ with $\G=\Gamma_1\times\Gamma_2\times \cdots \times \G_n$ and denote by $M= L(\G)$. If $t>0$ and $\La$ is an arbitrary group such that $L(\Lambda)=M^t $ then for every non-empty family $\mathcal G$ of subgroups of $\La$ there exists  $1\leq \ell\leq n$ such that, either:
 \begin{enumerate}
 \item ${L(\hat\Gamma_\ell)}^t \preceq_{M^t} L(\Sigma)$ for some $\Sigma\in\mathcal G$, or
 \item \label{1001} ${L(\G_{\ell})}^t \preceq_{M^t} L(\cup_s \Omega_s)$ where  $\Omega_s=C_\Lambda(\Sigma_s)$ and  $\Sigma_s \leqslant \Lambda$ is a descending sequence of subgroups such that  $\Sigma_s\notin \mathcal G$, for all $s$. 
 \end{enumerate}
\end{theorem}

\begin{proof}
Denote by $\set{u_\gamma}_{\gamma\in \Gamma}$, $\set{v_\lambda}_{\lambda\in \Lambda} $ the respective canonical $\Gamma,\Lambda $ group unitaries generating $M$.  Define $ \triangle : M^t\to M^t\bar{\otimes} M^t$ by $v_\lambda\mapsto v_\lambda\otimes v_\lambda $ and extend linearly.  $\triangle $ is clearly a $* $-homomorphism and $\triangle(M^t)\subset M^t\bar{\otimes} M^t $.  Now we may also view $M^t\bar{\otimes }M^t=L(\Gamma)^t\bar{\otimes} L(\Gamma)^t $. Let $k>t$ be an integer and let 
$p\in M^k$ be a projection of trace $s=tk^{-1}$. Since $ \Gamma$ is a product of $\Gamma_l $ and  $ \hat \Gamma_l$ we can write  $L(\La)=M^t= (M^k)^s =L(\hat\G_l)^k \bar\otimes L(\G_l)^s $, for every $1\leq l\leq n$. Thus we have
\begin{align*}
\triangle(L(\G_l)^s), \triangle(L(\hat\G_l)^k)\subset \triangle(L(\hat\G_l)^k \bar\otimes L(\G_l)^s )\subset L(\Gamma)^k\bar{\otimes} L(\Gamma)^k.
\end{align*}
By \cite[Lemma 15.3.3]{BO08}, $\Gamma\times\Gamma$ is bi-exact relative to the family \[\mathcal H = \{\G\times \hat \G_i, \hat \G_j\times \G \,:\,1\leq i,j\leq n\}.\] Hence, by \cite[Theorem 15.1.5]{BO08}\footnote{See also the main results in the seminal papers \cite{Oz03, Oz05} where these techniques where introduced or the more recent developments \cite[Theorem 3.2]{CS11} and \cite[Theorem 3.3]{BC14}.} we have that $\triangle(L(\hat \G_1)^t) \preceq_{M^k\bar\otimes M^k} M^k\bar\otimes L(\hat\G_i)^k$ or  $\triangle(L(\hat \G_1)^t) \preceq_{M^k\bar\otimes M^k} L(\hat\G_i)^k\bar\otimes M^k$ for some $1\leq i\leq n$. By consideration of symmetry we will only treat the case $\triangle(L(\hat \G_1)^t) \preceq_{M^k\bar\otimes M^k} M^k\bar\otimes L(\hat \G_i)^k$, where $1\leq i\leq n$ is fixed for the rest of the proof. 

By Popa's intertwining techniques, there exists $F\subset M^k$, finite, and $c'>0 $ so that 
\begin{align*}
\sum_{x\in F} \norm{ E_{M^k\bar\otimes L(\hat \Gamma_i)^k} (\triangle(y)\cdot (1\otimes x))}_2^2\geq c', \text{  for all }y\in \mathscr U(L(\hat\Gamma_1)^t),
\end{align*}
where $\|\cdot\|_2$ is induced by the trace $\tau$ on $M^k\bar \otimes M^k$. Writing  $y\in L(\hat \Gamma_1)^t\subset M^t $ as $y=\sum y_\lambda v_\lambda $ and using the previews equation we see that letting $c=\tau(p) c'$ we have
\begin{align}
\sum_{x\in F} \sum_{\lambda\in \Lambda} |y_\lambda|^2 \norm{ E_{L(\hat \Gamma_i)^k} (v_\lambda x)}_2^2\geq & c, \text{  for all }y\in \mathscr U(L(\hat \Gamma_1)^t). 
\end{align}
Now let $\mathcal{G}=\set{\Sigma_j\leqslant  \Lambda : j\in J} $ be an arbitrary family of subgroups of $\Lambda $.  We say that $S\subset \Lambda $ is \emph{small relative to $\mathcal{G} $} if there exist finite subsets $K,L\subset \Lambda $, $\mathcal{G}_0\subset \mathcal{G}$ so that $S\subset K\mathcal{G}_0 L$. Assume that for all $\Sg\in \mathcal G$ we have that $L(\hat \G_1)^t\not\preceq_{M^t} L(\Sg)$.  By Popa's intertwining technique it follows that for all  $\varepsilon>0$ and for each $S \subset\Lambda$ small relative to $\mathcal{G} $, there is a $y\in \mathscr U(L(\hat \Gamma_1)^t) $ so that $\sum_{\lambda\in S} |y_\lambda|^2<\varepsilon$. Then, indexing over all sets $S$ small relative to $\mathcal G$ it follows that for any $y\in \mathscr U(L(\hat \G_1)^t)$:
\begin{align}\label{2.1}
\sup_{\Lambda\setminus S}\norm{E_{L(\hat \Gamma_i)^k}(v_\lambda x)}_2^2\geq & \left[ \displaystyle\sum_{\lambda\in \Lambda\setminus S}|y_\lambda|^2\norm{E_{L(\hat \Gamma_i)^k}(v_\lambda x)}^2_2 \right]\left [\displaystyle\sum_{\lambda\in \Lambda\setminus S} |y_\lambda|^2 \right ]^{-1}, 
\end{align}
whence
\begin{align*}
\sum_{x\in F}\sum_{\lambda\in \Lambda\setminus S}|y_\la|^2\norm{E_{L(\hat \Gamma_i)^k}(v_\lambda x)}_2^2\geq &{\displaystyle \sum_{x\in F} \left[ \sum_{\lambda\in \Lambda}|y_\lambda|^2\norm{E_{L(\hat \Gamma_i)^k}(v_\lambda x)}^2_2-\sum_{\lambda\in S}|y_\lambda|^2\norm{E_{L(\hat \Gamma_i)^k}(v_\lambda x)}^2_2  \right]}\\
\geq & {\displaystyle c-|F|\max_{x\in F}\norm{x}_2^2\sum_{x\in S } |y_\lambda|^2}.
\end{align*}
Choosing $\varepsilon>0 $ sufficiently small and choosing $y\in \mathscr U(L(\hat \G_1)^k)$ so that $\sum_{\la\in \La\setminus S} |y_\la|^2< \ve$ we see 
\begin{align*}
\sum_{x\in F}\sum_{\lambda\in \Lambda\setminus S}|y_\la|^2\norm{E_{L(\hat \Gamma_i)^k}(v_\lambda x)}_2^2\geq {\displaystyle c-|F|\max_{x\in F}\norm{x}^2_2\varepsilon}>2^{-1} c>0.
\end{align*}

Using this together with (\ref{2.1}) we see that for every $S$ small with respect to $\mathcal{G} $ we have
\begin{align}\label{eq:SupCoefficientsLarge}
\sup_{\la\in \Lambda\setminus S}\sum_{x\in F}\norm{E_{L(\hat \Gamma_i)^k}(v_\lambda x)}_2^2\geq 2^{-1}(1-\ve)^{-1} c> 2^{-1}c.
\end{align}

Let $I$ be the directed set of all $S\subset \Lambda $ so that $S $ is small with respect to $\mathcal{G} $ and fix $\omega $ a cofinal ultrafilter on $I$.  Letting $\Theta=\Lambda\cap \theta\La \theta^{-1} $ for some $\theta\in \prod_{S\in \om}\Lambda\setminus S$, assume by contradiction $L(\Gamma_i)^t\not\preceq_{M^t} L(\Theta) $.  This implies $L(\Gamma_i)^s\not\preceq_{M^t} L(\Theta) $ and by Popa's intertwining techniques, there exists a sequence $\set{y_n}\subset \mathscr U(L (\Gamma_i)^s) $ so that 
\begin{align}\label{eq:ExpectationToZero}
\norm{E_{L(\Theta)} (xy_ny)}_2\to 0
\end{align}
as $n\to \infty $ for all $x,y\in M^t $; here $\|\cdot\|_2$ is induced by the trace on $M^t$.  Let $\mathcal K\subset L^2((M^k)^\omega)$ be the closed span of $\set{M^k pv_{\theta}pM^k} $ and $P_{\mathcal K}:L^2((M^k)^\omega)\to \mathcal K $ be the orthogonal projection onto $\mathcal K$.  Proceeding as in the proof of \cite[Theorem 3.1]{Io11}, equation \eqref{eq:ExpectationToZero} implies  $\generator{y_n\xi y_n^*,\eta}\to 0$ for all $\xi,\eta\in \mathcal K$. Also relation \eqref{eq:SupCoefficientsLarge} shows that $\sum_{x\in F}\norm{E_{(L (\hat\Gamma_i)^k)^\omega}(v_{\theta}x)}_2\geq c/2 $.  Take $x\in F$  so that $x'=E_{(L(\hat\Gamma_i)^k)^\omega}(v_{\theta} x)\neq 0 $ and $\xi_0= P_{\mathcal K}(x') $. Then
$\norm{v_{\theta}x-x'}_2<\norm{v_{\theta}x}_2 =\norm{x}_2$  which in turn gives $\norm{v_{\theta}x-\xi_0}_2<\norm{v_{\theta} x}_2 $, showing $ \xi_0\neq 0$. Since $[L(\G_i)^s,L(\hat\G_i)^k]=1$, we have $\norm{\xi_0}_2=\generator{y_ny_n^*\xi_0,\xi_0}=\generator{y_n\xi_0 y_n^*,\xi_0}\to 0$, contradicting $\xi_0\neq 0 $.  Thus $L(\Gamma_i)^t\preceq_{M^t}L(\Theta) $. Proceeding as in the proof of \cite[Theorem  3.1]{Io11} there exists a descending sequence of subgroups $\Sigma_s \leqslant \Lambda$ such that  $\Sigma_s\notin \mathcal G$ and  $\Theta=\cup_s \Omega_s $, where  $\Omega_s=C_\Lambda(\Sigma_s)$ and the conclusion and part (\ref{1001}) then follows. \end{proof}

\begin{cor}\label{int1}Let $\G_1,\G_2,\ldots ,\G_n\in \mathcal S$ with $\G=\G_1\times \G_2\times\cdots\times\G_n$ and denote by $M=L\G$. Let $t>0$ be a scalar and let $\La$ be an arbitrary group such that $M^t =L\La$. There exists a non-amenable subgroup $\Sigma<\La$ with non-amenable centralizer $C_\La(\Sigma)<\La$ and $1\leq \ell\leq n$ such that ${L(\hat\G_\ell)}^t\preceq_{M^t} L(\Sigma)$.  
\end{cor}
\begin{proof} From the assumptions it follows that $\La$ is a non-amenable, ICC group. Let $\mathcal G$ be the collection of all subgroups $\Sigma\leqslant \La$ with non-amenable centralizer $C_\Lambda(\Sigma)$ and notice it is non-empty ($\{1\}\in\mathcal G$). Applying Theorem \ref{comm1} for $\mathcal G$  there exists $1\leq \ell\leq n$ such that either: a) ${L(\hat\G_\ell)}^t \preceq_{M} L(\Sigma)$ for some $\Sigma\in\mathcal G$; or b) there exists a descending sequence of subgroups $\Sigma_s \leqslant \Lambda$ such that  $\Sigma_s\notin \mathcal G$ and  $L(\G_\ell) \preceq_{M} L(\cup_s \Omega_s)$, where  $\Omega_s=C_\Lambda(\Sigma_s)$. 

Assume b) holds. Since $\{\Sigma_s\}$ is a descending sequence of groups not belonging to $\mathcal G$  then  $\{\Omega_ s\}$ is an ascending sequence of amenable groups yielding that $\cup_s \Omega_s$ is amenable; thus, $L(\cup_s \Omega_s)$ is an amenable von Neumann algebra. By hypothesis $L(\G_k)$ is a non-amenable factor, so we must have $L(\G_\ell) \npreceq_{M^t} L(\cup_s \Omega_s)$, contradicting our assumption. 

Hence we must have a). Moreover, since $L(\hat\G_\ell)$ is a factor, whence has no amenable direct summand, $\Sigma$ is a non-amenable group, giving the desired conclusion. 
\end{proof}


\section{Proofs of the Main Results} In this section we use intertwining and solidity techniques to upgrade the existence of commuting nonamenable subgroups of the target group $\La$ to an honest direct product decomposition of $\La$. The bulk of the work goes into the proof of Theorem \ref{product1} which establishes the existence of a virtual product decomposition for $\La$. Before precisely stating this result, we begin with two preliminary solidity-type results on commuting subalgebras in group factors of products of groups in the class $\mathcal S_{nf}$.

\begin{lem}\label{reduction} For $n\geq 2$ let $\G_1,\G_2,\ldots,\G_n\in \mathcal S_{nf}$ with $\G=\G_1\times \G_2\times\cdots\times\G_n$ and denote by $M=L(\G)$. Let $t>0$ be a scalar and let $Q, A\subset M^t$ be commuting subalgebras such that $A$ is diffuse abelian and $Q=T_1\vee T_2\vee \cdots \vee T_{n-1}$, where $T_i\subset M^t$ are commuting, non-amenable II$_1$ subfactors. Then $Q\vee A\npreceq_{M^k} L(\hat \G_s)^k$ for all integers $1\leq s\leq n$ and $k\geq t$.   
\end{lem}

\begin{proof} Assume by contradiction $Q\vee A\preceq_{M^k} L(\hat \G_s)^{k}$, for some $1\leq s\leq n$ and $k\geq t$. Then there must exist a projection $p\in \mathscr P(Q\vee A)$, a scalar $k\geq t_1>0$, and a unital injective $\ast$-homomorphism $\phi: p(Q\vee A)p\ra L(\hat \G_s)^{t_1}$. Since $Q$ is a II$_1$ factor we have $\mathscr Z(Q\vee A)=A$ and denote by $E_A: Q\vee A \ra A$ the central trace. Since $0\neq E_A(p)$ there exist  $\mu>0$ and a projection $0\neq e\in A$ so that $E_A(pe)\geq \mu e$. Moreover, since $T_1$ is a II$_1$ factor there exists a projection $ r\in T_1 \subseteq Q$ so that $\tau_Q(r)=\mu$ and $E_A(re)=\mu e$. Thus $E_A(pe)\geq E_A(re)$ and since $E_A$ is a central trace there exists $w\in \mathscr I(Q\vee A)$ so that $pe\geq w^*w$ and $ww^*=re$.  Letting $u\in \mathscr U(Q\vee A)$ with $w=reu$ one can check that $\phi'=\phi\circ ad(u^*): re (Q\vee A)re \ra L(\hat \G_s)^{t_1}$ is an injective $\ast$-homomorphism; moreover, cutting $L(\hat \G_s)^{t_1}$ by a projection we can assume $\phi'$ is unital. By construction  $ A_1 =\phi'(Are)$ is diffuse abelian, and $T^1_i=\phi'(r eT_i r e)\subset L(\hat \G_s)^{t_1}$  are commuting, non-amenable II$_1$ subfactors. Moreover $A_1, Q_0=T^1_1\vee   \cdots \vee T^1_{n-1}\subseteq L(\hat \G_s)^{t_1}$ are commuting subalgebras. By \cite[Theorem 6.1]{CSU11}, one can find $s_1\in \{1,...,n\}\setminus \{s\}$ so that $Q_1\vee A_1\preceq_{L(\hat \G_s)^{k}}  L(\hat \G_{\{s,s_1\}})^{k}$, where $Q_1 =T^1_1\vee \cdots \vee T^1_{n-2}$. Applying the previous argument $n-2$ times one can find $1\leq s_{n-1}\leq n$, $k\geq t_{n-1}>0$, and commuting subalgebras $A_{n-1},Q_{n-2}\subset L(\G_{s_{n-1}})^{t_{n-1}}$ with $A_{n-1}$ diffuse abelian and $Q_{n-2}$ non-amenable II$_1$ factor. This however contradicts solidity of $L(\G_{s_{n-1}})$.  
\end{proof}

\begin{lem}\label{wc-emb} Let $\G_1,\G_2,\ldots,\G_n\in \mathcal S_{nf}$ with $\G=\G_1\times \G_2\times\cdots\times\G_n$ and denote by $M=L(\G)$. If $t>0$ is a scalar and $Q, A\subset M^t$ are commuting subalgebras such that $A$ is diffuse amenable then $[M^t:A\vee Q]_{PP}=\infty$.   
\end{lem}

\begin{proof} Assume by contradiction  $[M^t:A\vee Q]_{PP}<\infty$. Thus $Q$ is non-amenable and by \cite[Theorem 6.1]{CSU11} there exist a subset $F_1 \subset \{1,...,n\}$ with $|F_1|=n-1$ and an integer $t\leq k$ such that $A\preceq_{M^{k}} L( \G_{F_1})^{k}$. Using \cite[Proposition 2.4]{CKP14} one can find $k\geq t_1>0$ and diffuse amenable subalgebra $A_1\subset   L( \G_{F_1})^{t_1}$ so that $[L(\G_{F_1})^{t_1}: A_1 \vee Q_1]_{PP}<\infty$ where $Q_1= A_1'\cap L( \G_{F_1})^{t_1}$. Also note that $Q_1$ is non-amenable. Applying the same argument recursively, after $n-1$ steps  one can find $F_{n-1}\subset \{1,...,n\}$ with $|F_{n-1}|=1$, $t_{n-1}>0$, and diffuse amenable subalgebra $A_{n-1}\subset   L( \G_{F_{n-1}})^{t_{n-1}}$ so that $[L( \G_{F_{n-1}})^{t_{n-1}}: A_{n-1} \vee Q_{n-1}]_{PP}<\infty$, where $Q_{n-1}=A_{n-1}'\cap L( \G_{F_{n-1}})^{t_{n-1}}$. Since $L( \G_{F_{n-1}})$ is solid and $A_{n-1}$ is diffuse it follws $Q_{n-1}$ is amenable. Hence $A_{n-1}\vee Q_{n-1}$ is amenable and by \cite[Lemma 2.3]{PP86} and \cite[Proposition 2.4]{OP07} we conclude that $L( \G_{F_{n-1}})^{t_{n-1}}$ is amenable. By factoriality this further implies that  $\G_{F_{n-1}}$ is amenable which is a contradiction.
\end{proof}

\begin{theorem}\label{product1} Let $\G_1,\G_2,\ldots ,\G_n\in {\mathcal S}_{nf}$ be weakly amenable with $\G=\G_1\times \G_2\times\cdots\times\G_n$ and denote by $M=L(\G)$. Let $t>0$ be a scalar and let $\La$ be an arbitrary group such that $M^t =L(\La)$. Then there exist commuting, non-amenable, ICC subgroups $\Sigma_1,\Sigma_2<\La$ such that $[\La:\Sigma_1\Sigma_2]<\infty$.
\end{theorem} 
 
\begin{proof} The proof is quite technically involved, so it will be divided into a series of claims. Throughout proof we will be denoting by $\{u_\la\}_{\la\in \La}$ the canonical group unitaries implementing $L(\La)=M^t$.  For simplicity, denote by $P={L(\hat\G_n)}^t$, $N=L(\G_n)$, and notice $M^t=P\bar\otimes N$.  Using Corollary \ref{int1} there exists non-amenable subgroup $\Sigma<\La$ with non-amenable centralizer $C_\La(\Sigma)$ such that $P\preceq_{M^t} L(\Sigma)$. Thus, there exist $ p\in \mathscr P(P)$, $q\in \mathscr P(L(\Sigma))$,  $ v\in \mathscr I(M^t)$, and an injective $\ast$-homomorphism $\phi: pPp\ra qL(\Sigma) q$ so that 
\begin{equation}\label{1}
\phi(x)v=vx\text{ for all }x\in pPp.
\end{equation} For ease of notation let  $Q=\phi(pPp)$.

\begin{claim}\label{12} Without loss of generality we may assume that $Q\subset qL(\Sigma)q$ is a finite index inclusion of II$_1$ factors.
\end{claim} 
\noindent \emph{Proof of Claim \ref{12}}.  By \cite[Proposition 2.4]{CKP14} we can assume that $[qL(\Sigma) q:Q\vee (Q'\cap qL(\Sigma) q)]_{PP}<\infty$. 

We first show that $Q'\cap qL(\Sigma) q$ is purely atomic. Assume by contradiction there exists a diffuse corner $r(Q'\cap qL(\Sigma) q)r$. Hence there exists a diffuse, abelian subalgebra $A\subseteq r(Q'\cap qL(\Sigma) q)r$. However, since $C_\La(\Sigma)$ is non-amenable and $L(C_\Lambda(\Sigma))$ commutes with $Q\vee A$, then \cite[Theorem 6.1]{CSU11} implies  $Q\vee A\preceq_{M^k} {L(\hat \G_s)}^k$, for some integers $1\leq s\leq n$ and  $k\geq t$, contradicting Lemma \ref{reduction}. 

As a consequence we note that multiplying $q$ above by a nonzero minimal, central projection $q'$ of $Q'\cap qL(\Sigma) q$ we assume w.l.o.g.\ that $[qL(\Sigma) q:Q]_{PP}<\infty$. Moreover, letting $0\neq q'v$ instead of $v$ above, one can check that all previous equations still hold, including (\ref{1}). Since $Q$ is a II$_1$ factor and $[qL(\Sigma) q:Q]_{PP}<\infty$, Theorem \ref{basic} (\ref{fdim}) implies that $Q' \cap qL(\Sigma) q$ is finite dimensional. Since $\mathscr Z(qL(\Sigma) q)\subseteq Q' \cap qL(\Sigma) q$, $\mathscr Z(qL(\Sigma) q)$ is finite dimensional as well. Thus, multiplying $v$ above by a minimal, central projection of $\mathscr Z(qL(\Sigma) q)$ and using Theorem \ref{basic} (\ref{same}) the claim obtains. $\hfill\blacksquare$
\vskip 0.03in

\begin{claim}\label{12.5}  There is a projection $f\in \mathscr P(L(\Sigma)' \cap M^t)$ such that \[f(L(\Sigma) \vee (L(\Sigma)'\cap M^t)) f\subseteq f M^t f\] is a finite Jones index inclusion of II$_1$ factors.
\end{claim}

Performing the downward basic construction \cite[Lemma 3.1.8]{Jo81}, there exists a projection $e\in \mathscr P(qL(\Sigma) q)$ and a II$_1$ subfactor $R\subseteq Q\subseteq qL(\Sigma) q=\langle Q,e\rangle$ such that $[Q:R]=[qL(\Sigma) q:Q]$ and $Re=eL(\Sigma) e$. Keeping with the same notation, by relation (\ref{1}) the restriction $\phi^{-1}: R\ra pPp$ is  an injective $\ast$-homomorphism such that $T=\phi^{-1}(R)\subseteq pPp$ is a finite Jones index subfactor and 
\begin{equation}\label{4}
\phi^{-1} (y)v^*=v^*y\text{, for all }y\in R.
\end{equation}  
Let $\theta':Re\ra R$ be the $\ast$-isomorphism given by $\theta(xe)=x$. Letting $w^*=v^*e$, then $Re=eL(\Sigma) e$ together with (\ref{4}) shows that $\theta=\phi^{-1}\circ\theta': eL(\Sigma) e\ra pPp$ is an injective $\ast$-homomorphism satisfying $\theta(eL(\Sigma )e)=T$ and  
\begin{equation}\label{3}
\theta (y)w^*=w^*y\text{, for all }y\in eL(\Sigma) e.
\end{equation} 
Notice that $w^*w\in(T'\cap pPp)\bar\otimes  N$ and proceeding as in the proof of \cite[Proposition 12]{OP03} one can further assume that $w^*w\in\mathscr Z(T' \cap pPp) \bar\otimes N$. Since $[pPp:T]<\infty$ then $T' \cap pPp$ is finite dimensional and so is $\mathscr Z(T' \cap pPp)$. Thus, replacing the partial isometry $w$ by $wr_0\neq 0$, for some minimal projection  $r_0\in  \mathscr Z(T' \cap pPp)$, we see that all relations above still hold including relation (\ref{3}). Moreover, we can assume that $w^*w= z_1\otimes z_2$, for some nonzero projections $z_1\in   \mathscr Z(T' \cap pPp)$ and $z_2\in N$. Using relation (\ref{3}) we get  \begin{equation}\label{5}w^* L(\Sigma) w=\theta (eL(\Sigma) e)w^*w= Tz_1 \otimes z_2.
\end{equation}
Since $T\subseteq pPp$ is finite index inclusion of II$_1$ factors then by the local index formula \cite{Jo81} it follows $Tz_1\subseteq z_1Pz_1$ is a finite index inclusion of II$_1$ factors as well. 
Also, we have 
 \begin{equation}\label{6}
( w^* L(\Sigma) w)' \cap (z_1\otimes z_2)M^t (z_1\otimes z_2)= ((Tz_1)'\cap z_1Pz_1) \bar\otimes z_2Nz_2.
 \end{equation}
 
 Altogether, the previous relations imply that
 \begin{equation}\label{6'}\begin{split}
 Tz_1\bar \otimes z_2Nz_2&\subseteq Tz_1\vee (Tz_1'\cap z_1Pz_1) \bar\otimes z_2Nz_2 \\&= w^*L(\Sigma) w\vee  w^*( L(\Sigma)'\cap M^t) w\\
 &=w^* L(\Sigma) w\vee \left((w^* L(\Sigma) w)' \cap (z_1\otimes z_2)M^t (z_1\otimes z_2)\right)\\
 &\subseteq z_1Pz_1\bar\otimes z_2Nz_2.
 \end{split}
 \end{equation} 
Since $Tz_1\subseteq z_1Pz_1$ if a finite index inclusion of II$_1$ factors then by Theorem \ref{basic} (4)  so is $Tz_1\bar\otimes z_2Nz_2\subseteq z_1Pz_1\bar\otimes z_2Nz_2 $.  Letting $f=ww^*$ and $u\in \mathscr U(M^t) $ be a unitary such that $w^*= uww^*=uf$, relation (\ref{6'}) further implies that $f(L(\Sigma) \vee (L(\Sigma)'\cap M^t)) f=L(\Sigma) f \vee  f(L(\Sigma) '\cap M^t)f\subseteq f M^t f$
is an inclusion of finite Pimsner-Popa index. Moreover (\ref{6'}) and Theorem \ref{basic} (6) further imply that $\dim_\mathbb C (\mathscr Z(f(L(\Sigma) \vee (L(\Sigma)'\cap M^t)) f))\leq [ z_1Pz_1\bar\otimes z_2Nz_2 : Tz_1\bar\otimes z_2Nz_2]<\infty$. Thus shrinking $f$ if necessary, we can assume that  \begin{equation}\label{6}f(L(\Sigma) \vee (L(\Sigma)'\cap M^t)) f\subseteq f M^t f\end{equation} 
 is a finite index inclusion of II$_1$ factors. $\hfill\blacksquare$
\vskip 0.03in

For the remainder of the proof it will be convenient to introduce a notation.

\begin{notation}\label{omega} Denote by $\Omega=\{ \la \in \Lambda \,:\, | \mathcal O_{\Sigma} (\la)|< \infty \}$ and notice it is a subgroup of $\La$ normalized by $\Sg$.   
Let $\{\mathcal O_i \,:\, i \in \mathbb N\}$ be a (countable) enumeration of all the finite orbits of action by conjugation of $\Sigma$ on $\La$ and notice $\Omega= \cup_i \mathcal O_i$. Let $I_k \subset \mathbb N$ be an ascending sequence of finite sets such that $\cup_k I_k = \mathbb N$. Note that $\Omega_k := \langle \cup_{i\in I_k}\mathcal O_i\rangle $ is an ascending sequence
of subgroups of $\Omega$ normalized by $\Sg$ such that $\cup_k \Omega_k =\Omega$.
\end{notation} 

Next, we show the following:

\begin{claim}\label{13} $[\La:\Omega\Sg]<\infty$.
\end{claim}  
\noindent \emph{Proof of Claim \ref{13}}. By construction we have $L(\Sigma)'\cap M^t\subseteq L(\Omega)$ and hence $ f(L(\Sigma) \vee (L(\Sigma)'\cap M^t)) f\subseteq f L(\Omega\Sg)f$. Proceeding as in the proof of (\ref{6}), shrinking $f$ more if necessary,  we obtain that $f L(\Omega\Sg)f\subseteq fM^t f=fL(\Lambda) f$ is a finite index inclusion of II$_1$ factors. By Theorem \ref{basic} (5) it follows that the inclusion $f L(\Omega\Sg)f\subseteq fM^t f=fL(\Lambda )f$  has a finite Pimsner-Popa basis and hence the claim follows from Proposition \ref{fin-index1}.$\hfill\blacksquare$
\vskip 0.1in

\begin{claim}\label{9}$\Sigma \cap \Omega$ is finite.
\end{claim}

\noindent \emph{Proof of Claim \ref{9}}. 
Let $\mathcal O'_i =\mathcal O_i\cap \Sigma$ and
notice $\Sigma \cap \Omega= \cup_i \mathcal O'_i$. For each $k$ let $R_k = \langle \cup_{i\in I_k}\mathcal O'_i\rangle $ and notice it forms an ascending sequence
of normal subgroups of $\Sigma$ such that $\cup_k R_k =\Sigma \cap \Omega$ and $[\Sigma:\Sigma_k]<\infty$, where $\Sigma_k=C_\Sigma(R_k)$. Since $R_k\cap \Sigma_k$ is abelian and $[\Sigma:\Sigma_k]<\infty$ it follows that $R_k$ is virtually abelian; thus, $\Sigma \cap \Omega$ is a normal amenable subgroup of $\Sigma$. 

From {\bf Claim \ref{12}} we have obtained that $Q\subseteq qL(\Sg) q$ is a finite index inclusion of non-amenable II$_1$ factors. Denoting by $z=z(q)$, the central support of $q$ in $L(\Sg)$, we see that $L(\Sg) z$ is a non-amenable II$_1$ factor. Moreover there exists a scalar $s>0$ such that 
$(qL(\Sg) q)^s= L(\Sg) z$. By above $Q^s\subseteq (qL(\Sg) q)^s=L(\Sg) z$ is a finite index inclusion of non-amenable II$_1$ factors. Perform the basic construction $Q^s\subseteq L(\Sg) z\subseteq \langle L(\Sg) z, e_{Q^s}\rangle $ and notice that $\langle L(\Sg) z, e_{Q^s}\rangle= Q^\mu$ where $\mu=s[qL(\Sg) q:Q]^2$.

Finally we show $\Sigma \cap \Omega$ is finite. Notice the normalizer 
 $\mathscr N_{L (\Sg) z}(L(\Sigma \cap \Omega)z))''=L(\Sg) z$ has finite index in $Q^\mu$. From (\ref{1}) we have  $Q^\mu=N^{{t_1}/k}$ where $
 N=L(\hat \G_n) \bar\otimes M_k(\mathbb C)$ with  $t_1=\tau(p)\mu$ and $t_1\leq k\in N$.  Thus we can write $Q^\mu = N^r$ where $N=L(\hat \G_n) \bar\otimes M_k(\mathbb C)$, for some $k\in \mathbb N$ and $0<r< 1$.   $L(\Sigma) z$ is a factor  and $\G_s\in \mathcal S_{nf}$, for all $1\leq s\leq n-1$ then \cite[Theorem 1.6]{PV12} implies that $L(\Sg\cap \Omega)z\preceq^s_N L(\hat \G_{\{n,s\}})\bar \otimes M_k(\mathbb C)$, for every $1\leq s\leq n-1$. Thus using \cite[Lemma 2.5]{Va10} in the terminology therein we have $(L(\Sg\cap \Omega)z)_1 \subset_{approx} L(\hat \G_{\{n,s\}})\bar \otimes M_k(\mathbb C)$ for each $1\leq s\leq n-1$. Then by \cite[Lemma 2.7]{Va10} we further have that $(L(\Sg\cap \Omega)z)_1 \subset_{approx}  M_k(\mathbb C)$ and hence a corner of $L(\Sigma \cap \Omega)z$ is purely atomic. This together with Proposition \ref{fin-index1} implies that $\Sg\cap \Omega$ is finite. $\hfill\blacksquare$

 \begin{claim}\label{14} There exists $\ell$ such that $\Omega\subseteq \{\la\in \La \,:\, |\mathcal O_\Sg(\la)|\leq \ell\}$.
 \end{claim}
 
\noindent \emph{Proof of the Claim  \ref{14}}. By (\ref{6}) one can find a finite Pimsner-Popa basis $m_1,m_2,\ldots,m_s\in fM^tf$ such that for every $x\in fM^tf$ we have $x=\sum_i E_{f(L(\Sigma)\vee L(\Sigma)'\cap M^t)f}(xm^*_i)m_i$. Thus for all $x\in fM^tf$ we have  $\|x\|^2_{2,f}=\sum_i \|E_{f(L(\Sigma)\vee L(\Sigma)'\cap M^t)f}(xm^*_i)\|^2_{2,f}$ and hence for all $x\in fL(\Omega\Sigma)f$ we have $\|x\|^2_{2,f}=\sum_i \|E_{f(L(\Sigma)\vee L(\Sigma)'\cap M^t)f}(xn^*_i)\|^2_{2,f}$,  where $n_i= E_{fL(\Omega\Sigma)f}(m_i)$. 

Together with basic approximations of $n_i$'s, this show that for every $\ve>0$ there exist $c_\ve>0$ and a finite subset $L_\ve\subset \Omega$ such that for each $x\in fL(\Omega\Sigma)f$ we have 
\begin{equation}\label{7}\begin{split}
\tau(f)^{-1}\tau(fxfx^*)=\|x\|^2_{2,f} \leq\ve + c_\ve \tau(f)^{-1}\sum_{s\in L_\ve} \|E_{L(\Sigma)\vee L(\Sigma)'\cap M^t}(x u_s)\|^2_{2}. 
\end{split}
\end{equation}
Observe that for every $x\in L(\Omega)$ we have $E_{L(\Sigma)\vee L(\Sigma)'\cap M^t}(x )=E_{L(\Sigma\cap \Omega)\vee L(\Sigma)'\cap M^t}(x )$, noting that for every $x\in L(\Omega), a\in L(\Sigma), b\in L(\Sigma)'\cap M^t$ we have 

\begin{equation}\begin{split}
&\tau(E_{L(\Sigma)\vee L(\Sigma)'\cap M^t}(x)ab)=  \tau(xab)=\tau(aE_{L(\Sigma)}(bx))\\
&=\tau(aE_{L(\Sigma\cap \Omega)}(bx))=\tau(bxE_{L(\Sigma\cap \Omega)}(a))=\tau(bxE_{L(\Sigma\cap \Omega)\vee L(\Sigma)'\cap M^t}\circ E_{L(\Sigma\cap \Omega)}(a))\\
&=\tau(bxE_{L(\Sigma\cap \Omega)\vee L(\Sigma)'\cap M^t}\circ E_{L( \Omega)}(a))=\tau(bxE_{L(\Sigma\cap \Omega)\vee L(\Sigma)'\cap M^t}(a))\\
&=\tau(xE_{L(\Sigma\cap \Omega)\vee L\Sigma'\cap M}(ab))=\tau(E_{L(\Sigma\cap \Omega)\vee L(\Sigma)'\cap M^t}(x)ab)
\end{split}
\end{equation}
This formula together with (\ref{7}) gives that for every $\ve>0$ there exist $c_\ve>0$ and a finite subset $L_\ve\subset \Omega$ such that for all $x\in L(\Omega)$ we have 
 \begin{equation}\label{8}\begin{split}
 \tau(fxfx^*)\leq\ve + c_\ve  \sum_{s\in L_\ve} \|E_{L(\Sigma\cap\Omega)\vee (L(\Sigma)'\cap M^t)}(x u_s)\|^2_{2}. 
\end{split}
\end{equation}

By {\bf Claim \ref{9}}, the group $\Sigma\cap \Omega$ is finite and hence $L(\Sigma\cap\Omega)\vee (L(\Sigma)'\cap M^t)$ admits a finite Pimsner-Popa basis over $L(\Sigma)'\cap M^t$. Approximating the elements in this basis together with (\ref{8}) show that for every $\ve>0$ there exist $d_\ve>0$ and a finite subset $R_\ve\subset \Omega$ such that for every $x\in L(\Omega)$ we have 
 \begin{equation}\label{10}\begin{split}
 \tau(fxfx^*)\leq\ve + d_\ve \sum_{s\in R_\ve} \|E_{L(\Sigma)'\cap M^t}(x u_s)\|^2_{2}. 
\end{split}
\end{equation}
Setting up $S= \Omega$ and $\phi_\ve(x)=  d_\ve  \sum_{s\in L_\ve} \|E_{L(\Sigma)'\cap M^t}(x u_s)\|^2_{2}$, we see (\ref{10}) shows that part (\ref{0.1.1}) in Corollary  \ref{decay1} is satisfied. 

We will now show that if there are elements in $\Omega$ whose orbits under conjugation by $\Sigma$ have arbitrarily large size, this would imply that $\phi_\ve$ satisfies part (\ref{0.1.2}) in Corollary  \ref{decay1} as well; hence, we would have that $f=0$, a contradiction.

To this end fix $\delta>0$ and a finite subset $F\subset \Omega$.  For every $\sigma\in \Omega$ we have $E_{L(\Sigma)'\cap M^t}(u_\sigma)=|\mathcal O_\Sg(\sigma)|^{-1}\sum_{\la\in\mathcal O_\Sg(\sigma)}u_\la$. Since for all $t,v\in \Omega$ we have $|\mathcal O_\Sg(vt)|\leq |\mathcal O_\Sg(v)| |\mathcal O_\Sg(t)|$, the set $F R_\ve$ is  finite, and there exist elements in $\Omega$ whose orbits under conjugation by  $\Sigma$ have arbitrarily large size, then one can find $\sigma \in \Omega $ such that  $|\mathcal O_\Sg(\sigma^{-1}  \mu s)|\geq \delta^{-1} d_\ve |R_\ve|$, for all $\mu\in F$, $s\in R_\ve$. Thus for every $\mu \in F$ we have

 \begin{equation}\label{8.8}\begin{split}
\phi_\ve(u^*_\sigma u_\mu) &=  d_\ve \sum_{s\in R_\ve} \|E_{L(\Sigma)'\cap M^t}(u_{\si^{-1}\mu s})\|^2_{2} = d_\ve  \sum_{s\in R_\ve} |\mathcal O_\Sg(\sigma^{-1}  \mu s)|^{-1} \leq \delta.
\end{split}\end{equation} 
This finishes the proof of Claim \ref{14}. $\hfill\blacksquare$

\vskip 0.05in

\begin{claim}\label{24} For $k\in \bb N$, let $\Sigma_k := C_{\Sigma}(\Omega_k)$. There exist $\kappa\in \mathbb N$, $F\subset \Omega_{\kappa}$, and $C>0$ such that for every $\sigma\in \Sigma_{\kappa}$ there exists $s\in F$ such that $|\Cal O_\Omega(s\sigma)|\leq C$.  
 \end{claim}
\noindent \emph{Proof of the Claim  \ref{24}}. {\bf Claim \ref{14}} implies  $\|E_{L(\Sigma)'\cap M^t}(u_\la)\|^2_2= \||\mathcal O_\Sg(\la)|^{-1}\sum_{\mu\in\mathcal O_\Sg(\la)}u_\mu \|^2_2=|\mathcal O_\Sg(\la)|^{-1}\geq \ell^{-1}$, for all $\la \in \Omega$. Since $\Omega$ generates $L(\Omega)$, Popa's intertwining techniques  further imply $L(\Omega)\preceq_{L(\Omega)}L(\Sigma)'\cap M^t$. Thus, one can find $c\in \mathscr P(L(\Omega))$, $d\in \mathscr P(L(\Sigma)'\cap M^t)$, $w_1\in \mathscr I(L(\Omega))$, and a $\ast$-homomorphism $\alpha: cL(\Omega)c\ra d(L(\Sigma)'\cap M^t)d$ so that  \begin{equation}\label{17}\alpha(x)w_1=w_1 x\text{, for all }x\in cL(\Omega)c.\end{equation}  Let $u'\in \mathscr U(L(\Omega))$ and $r_0\in \mathscr P(cL(\Omega)c)$ so that $w_1=u'r_0$. Denoting by $B= \alpha(cL(\Omega)c)\subset d(L(\Sigma)'\cap M^t) d$ and $r_1=u'r_0(u')^*\in B'\cap cL(\Omega)c$,  relation (\ref{17}) implies $r_1L(\Omega)r_1=Br_1$. 

Fix $0<\ve<1$. Since $\Omega=\cup_k\Omega_k$ there exist $k_\ve\in \mathbb N$ and a projection $p_\ve\in L(\Omega_{k_\ve})$ so that 
\begin{equation}\label{18}\begin{split}
& \| p_\ve-r_1\|_2 \leq \ve.
\end{split}
\end{equation} 
Note that $[\Sg:\Sigma_{k_\ve}]<\infty$. Then (\ref{18}) together with $[B,L(\Sigma_{k_\ve})]=0$ shows that for all $x\in L(\Omega)$ and $y\in L(\Sigma_{k_\ve})$ we have $\| r_1 x r_1 y-yr_1 x r_1\|_2\leq 2\ve$. Altogether, these properties show that for every $\ve>0$ there exists $k_\ve \in \mathbb N$  such that for all $x\in \mathscr U(L(\Omega))$ and $y\in \mathscr U(L(\Sigma_{k_\ve}))$ we have $\tau (r_1 x r_1 x^*)=\| r_1 x  r_1\|^2_2 \leq 6\ve +   Re \tau((r_1 y)^* x r_1 y x^* )$. This further implies that
 \begin{equation}\label{19}\begin{split}\tau((ar_1 a^*)( b r_1  b^*))\leq 6\ve + Re \tau((a(r_1 y)^* a^*)( b r_1 y b^*))\text{ for all }a,b\in \mathscr U(L(\Omega)).
 \end{split}
\end{equation}
Consider a sequence of convex combinations $\xi_n= \sum^{k_n}_{i=1} \la_i b_i r_1 b^*_i$ which WOT-converges to $E_{L(\Omega)'\cap M^t}(r_1)$. Denote by $\eta_n=\sum^{k_n}_{i=1} \la_i b_i r_1 y b^*_i$, and notice that after passing to a subsequence we can assume that $\eta_n$ WOT-converges to an element $y_1\in (M^t)_1$.  Inequality (\ref{19}) implies that 
$\tau(ar_1 a^*\xi_n)\leq 6\ve + \tau(a(r_1 y)^* a^*\eta_n)$ for all $n\in \mathbb N$ and $a\in \mathscr U(L(\Omega))$. Taking limit over $n$ we get $\tau(ar_1 a^* E_{L(\Omega)'\cap M^t}(r_1))\leq 6\ve + \tau(a(r_1 y)^* a^* y_1)$ for all $a\in \mathscr U(L(\Omega))$. Since $a^*E_{L(\Omega)'\cap M^t}(r_1)=E_{L(\Omega)'\cap M^t}(r_1)a^*$ the previous inequality gives
\begin{equation}\label{19'}
\begin{split}
\|E_{L(\Omega)'\cap M^t}(r_1)\|^2_2  =\tau(ar_1 a^* E_{L(\Omega)'\cap M^t}(r_1)) \leq 6\ve + Re \tau(a(r_1 y)^* a^* y_1),\text{ for all } a\in \mathscr U(L(\Omega)).
\end{split}
\end{equation}

Consider a sequence of convex combinations $\zeta_n= \sum^{l_n}_{i=1} \mu_i a_i r_1 y a^*_i$ which WOT-converges to $E_{L(\Omega)'\cap M^t}(r_1 y)$. Using (\ref{19'}) we get $\|E_{L(\Omega)'\cap M^t}(r_1)\|^2_2  \leq 6\ve + Re \tau(\zeta^*_n y_1)$ for all n, whence passing to the limit over $n$ we have  $\|E_{L(\Omega)'\cap M^t}(r_1)\|^2_2   \leq 6\ve + Re \tau(E_{L(\Omega)'\cap M^t}((r_1 y)^*) y_1)$. Using the Cauchy-Schwarz inequality and $\|y_1\|_\infty \leq 1$ this further implies that 
\begin{equation}\label{19''}
\|E_{L(\Omega)'\cap M^t}(r_1)\|^2_2   \leq 6\ve + \|E_{L(\Omega)'\cap M^t}(r_1 y)\|_2,\text{ for all }y\in \mathscr U(L(\Sigma_{k_\ve} )).
\end{equation}

Using (\ref{18}) and (\ref{19''}), for every $\ve>0$ there exist $k_\ve \in \mathbb N$ and a finite subset $K_\ve\subset \Omega_{k_\ve}$ such that  
\begin{equation}\label{23}\begin{split}
\|E_{L(\Omega)'\cap M^t}(r_1)\|^2_2   \leq 8\ve +  \sum_{s\in K_\ve}\|E_{L(\Omega)'\cap M^t}(u_s y)\|_2,\text{ for all }y\in \mathscr U(L(\Sigma_{k_\ve} )).
\end{split}
\end{equation}

 We are now ready to prove the claim. Fix $\ve>0$.  Using (\ref{23}) there exist $k_\ve \in \mathbb N$ and a finite subset $K_\ve\subset \Omega_{k_\ve}$ such that for all $\sigma \in \Sigma_{k_\ve}$ we have \begin{equation}\label{24'}
\|E_{L(\Omega)'\cap M^t}(r_1)\|^2_2   \leq 8\ve +  \sum_{s\in K_\ve}\|E_{L(\Omega)'\cap M^t}(u_{s\sigma})\|_2=  8\ve +  \sum_{s\in K_\ve} |\Cal O_\Omega(s\sigma)|^{-1/2}.
\end{equation}
Assume by contradiction the claim does not hold; hence, for every $k\in \mathbb N$, $F\in \Omega_k$ and $C>0$ there exists  $\sigma\in \Sigma_{k}$ such that for all $s\in F$ we have  $|\Cal O_\Omega(s\sigma)|\geq C$. Since $K_\ve$ is finite one can find $\sigma \in \Sigma_{k_\ve}$ such that  $\sum_{s\in K_\ve} |\Cal O_\Omega(s\sigma)|^{-1/2}<\ve$, by (\ref{24'}) we have  $\|E_{L(\Omega)'\cap M^t}(r_1)\|^2_2\leq 9\ve$. As $\ve>0$ is arbitrary we get 
$r_1=0$ which is a contradiction. $\hfill\blacksquare$

\begin{claim}\label{25} There exist $R>0$ and a finite index subgroup $\Theta\leqslant\Sigma_{\kappa}$, such that $\Theta\subseteq \{ \la\in \La \,:\, |\Cal O_\Omega(\la)|\leq R\}$.  
 \end{claim}
 
\noindent \emph{Proof of the Claim  \ref{25}}. Using {\bf Claim \ref{24}} there exist $D>0$ and a finite cover  $\cup^r_{i=1} K_i=\Sigma_{\kappa} $  such that $\cup_i f_i K_i  \subseteq \{ \la\in \La \,:\, |\Cal O_\Omega(\la)|\leq D\}$ where $\{f_1,f_2,...,f_r\}=F$. Using  the inequality $|\mathcal O_\Omega(vt)|\leq |\mathcal O_\Omega(v)| |\mathcal O_\Omega(t)|$ for all $t,v\in \Omega$ one can find $s_i \in \Sigma_{\kappa}$ for $1\leq i\leq r$ such that $\cup_i s_i K_i  \subseteq \{ \la\in \La \,:\, |\Cal O_\Omega(\la)|\leq D^2\}$. Considering the subgroups  $\Theta_i=\langle s_iK_i \rangle  \leqslant\Sigma_{\kappa}$, the previous relations imply that $\cup^r_{i=1} \Theta_i\subseteq   \{ \la\in \La \,:\, |\Cal O_\Omega(\la)|<\infty\}$ and $\cup^r_{i=1} s^{-1}_i\Theta_i = \Sigma_{\kappa}$. Thus at least one the subgroups $\Theta:=\Theta_i\leqslant\Sigma_{\kappa}$ has finite index. Denoting by $\tilde K_i=K_i\cap \Theta$ we still have $\cup_i f_i\tilde K_i  \subseteq \{ \la\in \La \,:\, |\Cal O_\Omega(\la)|\leq D\}$ and $\cup^r_{i=1} \tilde K_i=\Theta $ and as before there exist $t_i\in \Theta$ so that  $\cup_i t_i \tilde K_i  \subseteq \{ \la\in \La \,:\, |\Cal O_\Omega(\la)|\leq D^2\}$. Since $\Cal O_\Omega(t^{-1}_i)<\infty$, the previous containment shows that $\Theta \subseteq \{ \la\in \La \,:\, |\Cal O_\Omega(\la)|\leq R\}$ where $R= D^2 \max_{i=1,r} | \Cal O_\Omega(t^{-1}_i)| $.   $\hfill\blacksquare$ 

\begin{claim}\label{26} If we denote by $\Omega' =\{ \la \in \La \,:\, |\Cal O_\Omega(\la)|<\infty\}$ then $\Theta \leqslant \Omega'$ has finite index. 
 \end{claim}

\noindent \emph{Proof of the Claim  \ref{26}}. Assume by contradiction that $\{ h_i\}$ is a infinite sequence of representatives of distinct right cosets of $\Omega'$ in $\Theta$. Since $\Theta \leqslant\Sigma$ and  $\Omega\Sg\leqslant \La$ are finite index and $\Sigma$ normalizes $\Omega$ if follows that $\Omega\Theta\leqslant \La$ is also finite index. Thus passing to an infinite subsequence of $\{h_i\}$ one can find $\la\in \La$ and $x_i\in \Omega\Theta$ so that $h_i=x_i\lambda$ for all $i\geq 1$. Thus $h_ih^{-1}_1\in \Omega\Theta$ for all $i\geq 2$. Hence one can find $\sigma_i\in \Theta$ and $\omega_i\in \Omega\cap \Omega'$ so that $ h_ih^{-1}_1=\sigma_i\omega_i$. From construction we have $\Theta \omega_i\neq \Theta \omega_j$ for all $i\neq j$. Since $\Theta\leqslant \Sigma$ one can check that $|\Cal O_{\Omega\Theta}(\omega_i)|<\infty$ for all $i$ and since $\Omega\Theta\leqslant\La$ has finite index we further have $|\Cal O_{\La}(\omega_i)|<\infty$ for all $i\geq 2$. This however contradicts the fact that $\La$ is ICC.  $\hfill\blacksquare$ 
\vskip 0.05in

We are now ready to complete the proof of the theorem. Combining  {\bf Claims \ref{9}} and {\bf \ref{26}} it follows that $\Omega'\cap \Omega$ is a finite group. Since $\Omega$ normalizes $\Omega '$ and $\Theta$ normalizes $\Omega$, for all $\sigma \in \Theta$, $\omega\in \Omega$ the words $\omega \sigma\omega^{-1}\sigma^{-1},\sigma\omega\sigma^{-1}\omega^{-1}\in \Omega'\cap \Omega$; hence the commutator $[\Theta, \Omega]\leqslant \Omega'\cap \Omega$. In particular $[\Theta, \Omega]$ is a finite normal subgroup of $\Omega\Theta$.  Since $\Omega\Theta$ has finite index in $\La$ and the latter is ICC it follows that $[\Theta, \Omega]=1$. Letting $\Sigma_1:=\Theta$ and $\Sigma_2:=\Omega$ we get our conclusion.\end{proof}

\begin{Remarks} If in the previous theorem we assume $\La$ is finitely generated, since $[\La: \Omega\Sigma]<\infty$, it follows that  $\Omega\Sigma$ is finitely generated as well.   Since from construction we have that $\Omega\Sigma_k$ is an increasing tower of subgroups such that $\cup_k \Omega\Sigma_k=\Omega\Sigma$, by finite generation, there exists $\ell \in \mathbb N$ such that $\Omega\Sigma_\ell =\Omega\Sigma$. Hence $\Sigma_\ell= C_\Sigma(\Omega_\ell)$ is a finite index subgroup of $\Sigma$ which commutes with $\Omega_\ell$. Thus $\Sigma_\ell\Omega_\ell$ has finite index in $\La$ and the conclusion of the theorem follows. Therefore in this situation one needs neither the weak amenability assumption nor the {\bf Claims 4.9-4.12}. This would be the case if we assume for instance that the groups $\G_i$'s  belong to class $S_{nf}$ and have property (T).    

\end{Remarks}


We require one more intermediate result before the proof of Theorem \ref{main3}.

\begin{theorem}\label{split1} Let $\G_1,\G_2,\ldots ,\G_n\in {\mathcal S}_{nf}$ which are all weakly amenable. Let $\G=\G_1\times \G_2\times\cdots\times\G_n$, and denote by $M=L(\G)$. Let $t>0$ be a scalar and let $\La$ be an arbitrary discrete group such that $M^t =L(\La)$. Then one can find subgroups $\Phi_1,\Phi_2<\La$ with $\Phi_1\times \Phi_2=\La$, a scalar $s>0$, a proper subset $F\subset\{1,\ldots,n\}$ and a unitary $v\in \mathscr U(M)$ such that $vL(\Phi_1)v^*= L(\G_F)^s$ and $vL(\Phi_2) v^*= (L(\G_{F^c}))^{t/s}$.     
\end{theorem}

\begin{proof}
From Theorem \ref{product1} there exist commuting, non-amenable, ICC subgroups $\Sigma_1,\Sigma_2<\La$ such that $[\La:\Sigma_1 \Sigma_2]<\infty$. Fix an integer $r>t$. Applying \cite[Theorem 6.1]{CSU11} there exist $1\leq k\leq 2$ and a smallest $1\leq i\leq n-1$ such that \begin{enumerate}
\item $L(\Sigma_k)\preceq_{M^r} {L(\G_F)}^r$ for some $F\subset\{1,\ldots, n\}$ with $|F|=i$; and 
\item $L(\Sigma'_k)\npreceq_{M^r} {L(\G_K)}^{r'}$, for all integers $r'\geq r$, all $K\subset \{1,\ldots, n \}$ with $|K|=i-1$, and all subgroups $\Sg_k'\leqslant \Sg_k$ with $[\Sigma_k:\Sigma'_k]<\infty$.\end{enumerate} For simplicity we assume throughout the proof that $k=1$. Denote by $M_1={L(\G_F)}^r$ and $M_2=L(\G_{F^c})^{t/r}$ and notice that $M^t=M_1\bar\otimes M_2$. Proceeding as in the proof of  \cite[Proposition 12]{OP03}  there exist a scalar $\mu>0$ and a partial isometry $u\in \mathscr I(M^t)$ satisfying $p=uu^*\in M^{1/\mu}_2$, $q=u^*u\in L(\Sigma_1)'\cap M^t$ and
\begin{equation}\label{4.1'}uL(\Sigma_1) u^*\subseteq M^\mu_1p.\end{equation} 

Consider the group $\Omega_2=\{ \la \in\La \,:\, |\mathcal O_{\Sigma_1}(\la) |<\infty \}$. Since $ \Sigma_2 \leqslant \Omega_2$ then from above it follows that $[\La: \Omega_2\Sg_1]<\infty$. Letting $\Omega_1=C_{\Sigma_1}(\Omega_2)$, it is a straightforward exercise to see that $\Omega_1,\Omega_2<\La$ are commuting, non-amenable, ICC subgroups  so that $[\La:\Omega_1 \Omega_2]<\infty$ and $[\Sg_1:\Omega_1]<\infty$. By construction we have $L(\Sigma_1) '\cap M^t \subseteq L(\Omega_2)$ and  the latter is a II$_1$ factor. Relation (\ref{4.1'}) gives $uL(\Omega_1) u^*\subseteq M^\mu_1p$. As in the proof of  \cite[Proposition 12]{OP03}, since $L(\Omega_2)$ and $M^{1/\mu}_2$ are factors, then for a large integer $m$ there exist $w_1, \ldots,w_m\in L(\Omega_2)$ and  $u_1,\ldots, u_m \in M^{1/\mu}_2$ partial isometries satisfying $w_i{w_i}^*=q'\leq q$, ${u_i}^*u_i=p'=uq'u^*\leq p$ for each $1\leq i\leq m$ and $\sum_j {w_j}^*w_j=1_{L(\Omega_2)}$, $\sum_j u_j{u_j}^*=1_{M^{1/\mu}_2}$. Combining with the above, one can check $v=\sum_j u_j uw_j\in M^t$ is a unitary satisfying $v L(\Omega_1)v^*\subseteq M^\mu_1$. This further implies that \begin{equation}\label{4.2'} v(L(\Omega_1)'\cap M^t)v^*\supseteq M^{1/\mu}_2.\end{equation}

Consider the groups $\Theta_2=\{ \la \in\La \,:\, |\mathcal O_{\Omega_1}(\la) |<\infty \}$ and $\Theta_1=C_{\Omega_1}(\Theta_2)$ and proceeding as before we have that $\Theta_1,\Theta_2 <\La$ are commuting, non-amenable, ICC subgroups  such that $[\La:\Theta_1 \Theta_2]<\infty$ and $[\Sg_1:\Theta_1]<\infty$.  Moreover, we have that $L(\Theta_2)\supseteq L(\Omega_1) '\cap M^t $ whence by (\ref{4.2'}) we have $vL(\Theta_2)v^*\supseteq M^{1/\mu}_2$. Since $M^t=M^\mu_1\bar \otimes M^{1/\mu}_2$ by \cite[Theorem A]{G96} we have that $vL(\Theta_2)v^*=B\bar \otimes M^{1/\mu}_2$, where $B\subseteq M^\mu_1$ is a factor. Thus $B$ and $vL(\Theta_1)v^*$ are commuting subfactors generating a finite index subfactor of $M^\mu_1$. Since  $ L(\Theta_1)\npreceq_M {L(\G_K)}^{r'}$, for $r'\geq r$ and all $K\subset \{1,\ldots, n \}$ with $|K|=i-1$ and since $vL(\Theta_1)v^*$ is non-amenable, it follows from \cite[Theorem 6.1]{CSU11} that $B$ is amenable. Also if we assume that $B$ is diffuse then by Lemma \ref{wc-emb} one contradicts the solidity of $L(\G_i)$; thus, $B$ is a completely atomic factor whence $B=M_k(\mathbb C)$, for some integer $k$. Altogether, this shows that $vL(\Theta_2)v^*= M^{s}_2$ where $s=k/\mu$. Since $M=M^{1/s}_1\bar\otimes M^s_2$ we also have $v(L(\Theta_2)'\cap M)v^*= M^{1/s}_1$. Let $\Phi_1=\{ \la \in\La \,:\, |\mathcal O_{\Theta_2}(\la) |<\infty \}$ and since $\Theta_2$ is ICC it follows that $ \Phi_1\cap\Theta_2=1$.  By construction we have that $vL(\Phi_1)v^*\supseteq v(L(\Theta_2)'\cap M)v^*= M^{1/s}_1$; hence, using \cite[Theorem A]{G96} again we get $vL(\Phi_1)v^*=A\bar \otimes v(L(\Theta_2)'\cap M)v^*$ for some $A \subseteq vL(\Theta_2)v^*$. In particular, we have that $A=v L(\Phi_1)v^*\cap v L(\Theta_2)v^* =vL(\Phi_1\cap \Theta_2)v^*= \mathbb C 1$. This further implies that $vL(\Phi_1)v^*=v(L(\Theta_2)'\cap M) v^*$ whence $\Phi_1$ and $\Theta_2$ are commuting, non-amenable subgroups of $\La$ such that $\Phi_1\cap\Theta_2=1$, $\Phi_1\Theta_2=\La$, $vL(\Phi_1)v^*=  M^{1/s}_1$, and $vL(\Theta_2)v^*= M^{s}_2$. Letting $\Phi_2=\Theta_2$ we get the desired conclusion.\end{proof}

\begin{theorem}[Theorem \ref{main3}] Let $\G_1,\G_2,\ldots ,\G_n\in {\mathcal S}_{nf}$ be weakly amenable groups with $\G=\G_1\times \G_2\times\cdots\times\G_n$ and denote by $M=L(\G)$. Let $t>0$ be a scalar and let $\La$ be an arbitrary group such that $M^t =L(\La)$.  Then one can find subgroups $\La_1,\La_2,\ldots,\La_n<\La$ with $\La_1\times \La_2\times \cdots \times \La_n=\La$, scalars $t_i>0$ with $t_1t_2\cdots t_n=t$, and a unitary $u\in \mathscr U(M)$ such that $uL(\La_i)u^*= {L(\G_i)}^{t_i}$ for all $1\leq i\leq n$.     
\end{theorem}
\begin{proof} We proceeding by induction on $n$. Since $n=2$ follows directly from Theorem \ref{split1} it remains to only show the induction step. So assume the statement is satisfied for any collection of $k\leq n-1$ groups belonging to $\mathcal S_{nf}$. By Theorem \ref{split1} one can find subgroups $\La'_1,\La'_2<\La$ with $\La'_1\times \La'_2=\La$, a scalar $s>0$, a proper subset $F\subset\{1,\ldots,n\}$, and a unitary $v\in \mathscr U(M^t)$ such that $vL(\La'_1)v^*= L(\G_F)^s$ and $vL(\La'_2) v^*= L(\G_{F^c})^{t/s}$. Since $vL(\La'_1)v^*= L(\G_F)^s$ then by the induction hypothesis there exist subgroups $\La_i<\La'_1$, $i\in F$ with $\times_{i\in F} \La_i=\La'_1$, scalars $t_i>0$ with $\Pi_{i\in F}  t_i=s$, and $u_1\in \mathscr U(L(\G_F)^s)$ such that $u_1L(\La_i )u_1^*= {L(\G_i)}^{t_i}$, for all $i\in F$. Similarly, one can find subgroups $\La_i<\La'_2$, $i\in F^c$ with $\times_{i\in F^c} \La_i=\La'_2$, scalars $t_i>0$ with $\Pi_{i\in F^c}  t_i=t/s$, and  $u_2\in \mathscr U(L(\G_{F^c})^{t/s})$ such that $u_2L(\La_i) u_2^*= L(\G_i)^{t_i}$, for all $i\in F^c$. Altogether, the previous relations imply $\times^n_{j=1} \La_j=\La'_1\times \La'_2=\La$ and $\Pi^n_{j=1} t_j= (\Pi_{i\in F} t_i)( \Pi_{i\in F^c} t_i)=s \cdot t/s=t$. Moreover, letting $u=(u_1\otimes u_2)v$ we get $uL(\La_i)u^*= {L(\G_i)}^{t_i}$ for all $1\leq i\leq n$, as desired.\end{proof}



\subsection{Remarks on the use of weak amenability}

Notice that in the proof of Theorem \ref{main3} the assumption that the $\G_i$'s are weakly amenable is used only in the proof of {\bf Claim \ref{9}}, class $\mathcal S_{nf}$ being sufficient for all the other steps of the proof. In fact when we are dealing with only two groups the weak amenability assumption can be dropped.

\begin{theorem}\label{product2}Let $\G_1,\G_2\in \mathcal S$ with $\G=\G_1\times \G_2$ and denote by $M=L(\G)$. Let $t>0$ be a scalar and let $\La$ be an arbitrary group such that $M^t =L(\La)$. Then there exist commuting, non-amenable, ICC subgroups $\Sigma_1,\Sigma_2<\La$ such that $[\La:\Sigma_1 \Sigma_2]<\infty$.
\end{theorem} 
 
\begin{proof} The proof is identical with the proof of Theorem \ref{product1}. Thus we borrow the same notations and set-up as there. We only include a proof for Claim \ref{9} which bypasses the usage of weak amenability, the rest of the proof following the same arguments. Under the same notations as in the proof of Theorem \ref{product1} we show that:

\begin{claim}\label{9'''}$\Sigma \cap \Omega$ is finite.
\end{claim}

Let $\mathcal O'_i =\mathcal O_i\cap \Sigma$ and
notice $\Sigma \cap \Omega= \cup_i \mathcal O'_i$. For each $k$ let $R_k = \langle \cup_{i\in I_k}\mathcal O'_i\rangle $ and notice it forms an ascending sequence
of normal subgroups of $\Sigma$ such that $\cup_k R_k =\Sigma \cap \Omega $ and $[\Sigma:\Sigma_k]<\infty$, where $\Sigma_k=C_\Sigma(R_k)$. Since $R_k\cap \Sigma_k$ is abelian and $[\Sigma:\Sigma_k]<\infty$ it follows that $R_k$ is virtually abelian; thus, $\Sigma \cap \Omega$ is a normal amenable subgroup of $\Sigma$. 

In the first part of the proof of Theorem \ref{main3} (see Claim \ref{12}) we have obtained that $Q\subseteq qL(\Sg) q$ is a finite index inclusion of non-amenable II$_1$ factors. Denoting by $z=z(q)$, the central support of $q$ in $L(\Sg)$, we see that $L(\Sg) z$ is a non-amenable II$_1$ factor. Moreover there exists a scalar $s>0$ such that 
$(qL(\Sg) q)^s= L(\Sg) z$, so $Q^s\subseteq (qL(\Sg) q)^s=L(\Sg) z$ is a finite index inclusion of non-amenable II$_1$ factors. Perform the basic construction $Q^s\subseteq L(\Sg) z\subseteq \langle L(\Sg) z, e_{Q^s}\rangle $, and notice that $\langle L(\Sg) z, e_{Q^s}\rangle= Q^\mu$ where $\mu=s[qL(\Sg) q:Q]^2$. 

To begin we argue that each $R_k$ is finite. Since $C_k:=R_k\cap \Sigma_k\leqslant R_k$ has finite index it suffices to show that $C_k$ is finite. From construction note that \begin{equation}\label{1000}
L(C_k)\subseteq \mathscr Z(L(\Sg_k))\subseteq L(\Sg_k)'\cap L(\Sg).
\end{equation}
By passing to a finite index subgroup we can further assume w.l.o.g.\ that $\Sg_k\lhd \Sg$ is normal and $[\Sg:\Sg_k]=r$. Fix $\g_1,\g_2,...,\g_r\in \Sg$ a complete set of representatives for the cosets of $\Sg_k$ in $\Sg$. One can check the map $E_{\mathscr Z(L(\Sg))}: L(\Sg_k)'\cap L(\Sg)\ra \mathscr Z(L(\Sg))$ given by $E_{\mathscr Z(L(\Sg))}(x)=r^{-1}\sum_i u_{\g_i} xu^*_{\g_i}$ for all $x \in L(\Sg_k)'\cap L(\Sg)$ is  a trace-preserving conditional expectation satisfying $\|E_{\mathscr Z(L(\Sg))}(x)\|^2_2 \geq r^{-1}\|x\|^2_2$; hence $[L(\Sg_k)'\cap L(\Sg):\mathscr Z(L(\Sg))]_{PP}\leq r$. By Theorem \ref{basic} (2) we have $[L(\Sg_k)'\cap L(\Sg) z:\mathscr Z(L(\Sg)) z]_{PP}\leq r$ and since $L(\Sg)z$ is a factor it follows that $L(\Sg_k)'\cap L(\Sg) z$ is finite dimensional. Hence there exists $z_0\in \mathscr P(L(\Sg_k)'\cap L(\Sg))$ so that $L(\Sg_k)'\cap L(\Sg) z_0=\mathbb Cz_0$. Combining with (\ref{1000}) we have $z_0\in \mathscr P(L(C_k)'\cap L(\Sg))$ and $L(C_k) z_0=\mathbb C z_0$. By functional calculus one can find $z_1\in \mathscr P (\mathscr Z(L(C_k)))$ so that $L(C_k) z_1=\mathbb C z_1$ and Corollary \ref{diffcorn} entails $C_k$ is finite.

We now show $\Sigma \cap \Omega$ is finite. Proceeding by contradiction assume that $\Sigma\cap \Omega$ is infinite whence $L(\Sigma \cap \Omega)z\subset L(\Sg) z$ is a diffuse subalgebra. Also notice that $z\in \mathscr Z (L(\Sg))\subseteq L(\Sg\cap \Omega)$ and $\mathscr Z(L(\Sg))=\mathbb c z$. Let $\mathscr G = \{u_\sigma z : \sigma\in \Sigma\}\subset \mathscr N_{L(\Sigma)z}(L(\Sg\cap\Omega)z)$. We argue that the natural action by conjugation $\mathscr G\curvearrowright L(\Sigma \cap \Omega)z$ is weakly compact. For every $n\in \mathbb N$ consider the selfadjoint element $\xi_n=|R_k|^{-1/2}\sum_{a\in R_k} (u_a z)\otimes (\overline{u_a z}) \in L(R_k)z \bar\otimes L(R_k)z$. Since $R_k$ is normal in $\Sg$ we have 
\begin{equation}\label{1001}
(u_\g z)\otimes (\overline{u_\g z}) \xi_n=\xi_n (u_\g z)\otimes (\overline{u_\g z}), \text{ for all }\g \in \Sg.
\end{equation} 
Also one can check for all $a\in R_k$  and $l\geq k$ we have $(u_a z)\otimes( \overline{u_a z}) \xi_l=\xi_l $. This implies
\begin{equation}\label{1002}
\lim_n\|(u_a z)\otimes( \overline{u_a z}) \xi_n-\xi_n \|_{z,2}=0, \text{ for all }a\in \Sg\cap \Omega.
\end{equation} 
Here $\|\cdot\|_{z,2}$ is the $2$-norm induced by the trace $\tau_z(y)=\tau(yz)\tau^{-1}(z)$ on $L(\Sg)z$, where $\tau$ is the canonical trace on $L(\Sg)$.

For all $k\in \mathbb N$ and  $x\in L(\Sg)$ we have  the following basic calculations:
\begin{equation*}\label{1003}
\begin{split} \langle (xz)\otimes z \xi_k,\xi_k\rangle& = |R_k|^{-1}\sum_{a,b\in R_k} \langle xu_a z,u_bz\rangle \langle \overline{u_a z},\overline{u_bz}\rangle=|R_k|^{-1}\sum_{a,b\in R_k} \tau(xu_a zu_{b^{-1}}) \tau(zu_{a^{-1}b})\\
&=|R_k|^{-1}\sum_{a,s\in R_k} \tau(xu_a zu_{s^{-1} a^{-1}}) \tau(zu_s)=|R_k|^{-1}\sum_{a\in R_k} \tau( xu_a z(\sum_{s\in R_k}\tau(zu_{s})u_s^{-1}) u_{a^{-1}})\\
&=|R_k|^{-1}\sum_{a\in R_k} \tau( xu_a zE_{L(R_k)}(z) u_{a^{-1}})=\tau( x zE_{L(R_k)}(z)).
\end{split}
\end{equation*} 
Since $\cup_k R_k = \Sigma \cap \Omega$ we further have that 
\begin{equation}\label{1004} 
\lim_n \langle (xz)\otimes z \xi_n,\xi_n\rangle_{z}=\lim_n \tau_z( x zE_{L(R_n)}(z))=\tau_z(x), \text{ for all }x\in L(\Sg).
\end{equation} A similar computation shows that \begin{equation}\label{1005} 
\lim_n \langle z\otimes (\overline{xz}) \xi_k,\xi_k\rangle_{z}=\tau_z(x), \text{ for all }x\in L(\Sg).
\end{equation}
  From (\ref{1}) we have  $Q^\mu=L(\G_1)^{t_1}$ with $t_1=\tau(p)\mu$, and from above we get that $L(\Sg)z\subseteq L(\G_1)^{t_1}$ is a finite index inclusion of  II$_1$ factors. Since $\G_1\in \mathcal S_{nf}$ then the relations (\ref{1001})-(\ref{1004}) together with the same proof as in \cite[Theorem 6.1]{CSU11} show that $\mathscr G'' = L(\Sigma)z$ is amenable. However, this contradicts the non-amenability of $L(\G_1)$. Hence $\Sigma \cap \Omega$ is finite.
\end{proof}

\begin{theorem}[Corollary \ref{main1}]\label{split2} Let $\G_1,\G_2\in \mathcal S_{nf}$ with $\G=\G_1\times \G_2$ and denote by $M=L(\G)$. Let $t>0$ be a scalar and let $\La$ be an arbitrary group such that $M^t =L(\La)$. Then one can find subgroups $\La_1,\La_2<\La$ with $\La_1\times \La_2=\La$, a scalar $s>0$, a proper subset $F\subset\{1,\ldots,n\}$ and a unitary $v\in \mathscr U(M)$ such that $vL(\La)_1v^*= L(\G_1)^s$ and $vL(\La)_2 v^*= L(\G_2)^{t/s}$.     
\end{theorem}

\begin{proof} If follows from Theorem \ref{product2}, proceeding as in the proof of Theorem \ref{split1}
\end{proof}


\subsection{Von Neumann algebras generated by lattices} In this section we show how the main result can be used to give new families of examples of ICC discrete groups which are measure equivalent, but whose group von Neumann algebras are not stably isomorphic. The first such examples were constructed by the first author and Ioana \cite{CI11}. We recall that discrete groups $\G$ and $\La$ are said to be \emph{measure equivalent} if there exists an essentially free action $\G\times\La\ca (\Omega, m)$ on a standard $\sigma$-finite measure space such that the restriction of the action to each of $\G$ and $\La$ admits a Borel fundamental domain of finite measure. The prototypical examples of measure equivalent discrete groups are pairs of lattices in semisimple Lie groups: see \cite{Fu99a} for a thorough treatment of measure equivalence. 

In order to establish the result, we will need to make use of a celebrated theorem of Margulis in order to distinguish the algebraic structure of certain lattices in the same ambient Lie group. We briefly recall that if $G$ is a semisimple Lie group without compact factors, then a lattice $\G < G$ is said to be \emph{reducible} if $G$ admits semisimple factors $G_1$ and $G_2$ so that setting $\G_1 := \G\cap G_1$ and $\G_2 := \G\cap G_2$ we have that $\G_1\G_2$ is a finite index subgroup of $\G$. For example, a finite product of lattices is a reducible lattice in the product of the respective ambient Lie groups. A lattice is \emph{irreducible} if it is not reducible. For instance, the natural inclusion of $\mathbb Z[\sqrt 2]$ as a lattice in $\mathbb R\times \mathbb R$ induces an irreducible lattice inclusion of $\PSL(2, \mathbb Z[\sqrt 2])$ into $\PSL(2,\mathbb R)\times \PSL(2,\mathbb R)$. Note that if a group $\G$ virtually decomposes as a direct product of infinite groups, e.g., $\G$ is a reducible lattice, then $\G$ admits an infinite normal subgroup with infinite index.  However it is a consequence of the following famous theorem of Margulis \cite{Ma79, Z84} that irreducible lattices can admit no direct product decomposition:

\begin{theorem}[Normal Subgroup Theorem] If $\G< G$ is an irreducible lattice in a higher-rank semisimple Lie group $G$ with no compact factors then $\G$ is virtually simple, i.e., any normal subgroup of $\G$ is either finite or has finite index in $\G$.
\end{theorem}

\noindent Note that by the Borel density theorem, if $\G$ is a lattice in a semi-simple Lie group with no compact factors and trivial center, then $\G$ is ICC, whence if irreducible cannot be decomposed as \emph{any} direct product of groups. 

The following corollary is a direct application of the Normal Subgroup Theorem to Theorem \ref{main3} and Corollary \ref{main1}.

\begin{cor}[Corollary \ref{main2}] If $\La$ is an irreducible lattice in a higher rank semisimple Lie group, then $L(\La)$ is neither isomorphic to a factor $L(\G_1\times \G_2)$ where $\G_1,\G_2$ are groups in the class $\mathcal S_{nf}$, nor is it isomorphic to a factor of the form $L(\G_1\times\dotsb\times \G_n)$ where each $\G_i\in \mathcal S_{nf}$ and is weakly amenable. 
\end{cor} 
We apply this corollary in the following more specific situation. Let $G_1,\dotsc, G_n$ be non-compact, real simple Lie groups of rank one with trivial center, and let $\G_1,\dotsc, \G_n$ be respective lattices. It follows from \cite{CH89} that each $\G_i$ is weakly amenable and from \cite{Oz06} that each $\G_i$ belongs to the class $\mathcal S_{nf}$. Thus if $\La < G_1\times\dotsb\times G_2$ is an irreducible lattice, then $\La$ and $\G_1\times\dotsb\times\G_n$ are measure equivalent groups whose associated group factors are not stably isomorphic.

\end{document}